\documentclass[12pt]{amsart}
\usepackage{amsmath,amssymb,latexsym}
\usepackage{fullpage}
\usepackage{enumerate}

\theoremstyle{plain}
\newtheorem{theorem}{Theorem}
\newtheorem{proposition}{Proposition}
\newtheorem{lemma}{Lemma}
\newtheorem{definition}{Definition}
\newtheorem{conjecture}{Conjecture}

\begin{document}

\title{Descent and Theta Functions for Metaplectic Groups}
\author{Solomon Friedberg}
\author{David Ginzburg}
\address{Department of Mathematics, Boston College, Chestnut Hill MA 02467-3806, USA}
\email{solomon.friedberg@bc.edu}
\address{School of Mathematical Sciences, Tel Aviv University, Ramat Aviv, Tel Aviv 6997801,
Israel}
\email{ginzburg@post.tau.ac.il}
\thanks{This work was supported by the US-Israel Binational Science Foundation,
grant number 201219, and by the National Security Agency, grant number
H98230-13-1-0246 and the National Science Foundation, grant number 1500977 (Friedberg).}
\begin{abstract} 
There are few constructions of square-integrable automorphic functions on metaplectic groups.  Such functions may
be obtained by the residues of certain Eisenstein series on covers of groups, ``theta functions," but 
the Fourier coefficients of these residues are not well-understood, even for low degree covers of $GL_2$.
Patterson and Chinta-Friedberg-Hoffstein proposed conjectured relations for the Fourier
coefficients of the $GL_2$ quartic and sextic theta functions (resp.), each obtained from a conjectured equality
of non-Eulerian Dirichlet series.  
In this article we propose
a new framework for constructing specific $L^2$ metaplectic functions and for understanding these conjectures: descent integrals.
We study descent integrals which begin with theta functions on covers of larger rank classical groups and use
them to construct certain $L^2$ metaplectic functions on covers related to $GL_2$.  
We then establish information about the Fourier coefficients of these metaplectic automorphic functions, properties
which are consistent with the conjectures of Patterson and Chinta-Friedberg-Hoffstein.
In particular, 
we prove that Fourier coefficients of the descent functions are arithmetic for infinitely many primes $p$.
We also show that they generate a representation with non-zero projection to the space of theta.
We conjecture that the descents may be used to realize the quartic and sextic theta functions.  Moreover, 
this framework suggests
that each of the conjectures of Patterson and Chinta-Friedberg-Hoffstein
is the first in a series of relations between certain Fourier coefficients of two automorphic
forms on different covering groups. 
\end{abstract}
\maketitle

\section{Introduction}

Let $n>1$ and $F$ be a number field containing a full group of $n$th roots of unity, $\mu_n$.
For $G$ in a broad class of algebraic groups including
all Chevalley groups, there is a family of central simple extensions of the adelic points of $G$
by $\mu_n$: the metaplectic groups. If $\widetilde{G}$ is a metaplectic group, the rational points $G(F)$ embed discretely
and one may study automorphic functions for $\widetilde{G}$.  The first examples of such functions,
Eisenstein series, may be constructed as in the case of $G$ itself.  However, these are not square-integrable.
Metaplectic cusp forms are square-integrable, but almost nothing is known about their Whittaker coefficients.
Here we construct square-integrable functions on $G(F)\backslash \widetilde{G}$ in specific cases
and use the constructions to give information about their Whittaker coefficients.

A first way to obtain square-integrable functions on metaplectic groups is as the residues of certain Eisenstein series.  
These residues, ``higher theta functions,"
 are difficult to study, and even for $G=GL_2$, for covers of degree $n$ greater than 3
very little is known.
The metaplectic
Eisenstein series for $G=GL_2$, $n>2$ were introduced by Kubota \cite{Ku}, who showed in the 1960s that their Fourier
coefficients are infinite sums of $n$-th order Gauss sums.  Hence one can not approach
the problem of determining their residues directly from this expansion.
When $n=3$, Patterson determined the cubic theta function in full in the 1970s,
and showed that its Fourier coefficient at a good prime $p$ was a Gauss sum attached to the cubic residue symbol
$(*/p)_3$ \cite{Pa}.
However, for $G=GL_2$, $n>3$, such a determination is not known, and
the difficulty of this problem has a representation-theoretic origin: the Whittaker functional
of the local theta representation is not unique.

In spite of this non-uniqueness, there are two conjectures concerning the Fourier coefficients of the higher theta functions.
In the early 1980s Patterson considered the case $n=4$ \cite{Pa2}; his
conjecture states that the square of the Fourier coefficient of the biquadratic theta function at a good prime $p$
is twice a normalized Gauss sum formed from a biquadratic residue symbol $(*/p)_4$.  Numerical and additional theoretical
work of Eckhart and Patterson  \cite{Eck-Pat} has supported this conjecture,
but further progress has been hard to come by;
a possible approach using Rankin-Selberg integrals has been offered by Suzuki \cite{Su1},
who was able to prove results about the arguments but not the modulus of the coefficients.
More recently, Chinta, Friedberg, and Hoffstein \cite{C-F-H}
conjectured a relation (the CFH conjecture)
between the Fourier coefficients at squares of the theta function on the six-fold
cover of $GL_2$ and cubic Gauss sums, or equivalently the Whittaker coefficients of the cubic theta
function.  This conjecture has recently been studied numerically and refined  by
Br\"oker and Hoffstein \cite{Br-Ho}, working over $\mathbb{Q}(\exp(2\pi i/6))$. 
Each of these conjectures can be reformulated as a conjectured equality of two Dirichlet series which are not Eulerian
but have analytic continuation and functional equation.  
Complementing these numerical studies for number fields, 
some progress has been made for function fields: in the early 1990s, 
Hoffstein \cite{Ho} proved Patterson's Conjecture over the rational function field, and
the CFH conjecture was established in this case as well in \cite{C-F-H}.  However,
these proofs are by explicit computation of rational functions, and do not shed light
on why the conjectures should be true for number fields (or more general function fields). Moreover, these conjectures
have not been fit into a general framework, and it has not been apparent whether, even if proved, they would
be isolated facts or the first instances of a more general theory.

The goals of this paper are to provide a new approach to constructing 
specific square-integrable automorphic functions on metaplectic 
groups, by using descent integrals for metaplectic groups, and to provide as well a new
framework for viewing the conjectures of Patterson and Chinta-Friedberg-Hoffstein, namely through such descent integrals.   
Indeed, we construct $L^2$ functions on these
covering groups modulo the discrete subgroup 
whose properties closely match the conjectured properties of the quartic and sextic theta functions
(in particular, the relevant Fourier coefficients at a set of primes of positive density
are algebraic).
Unfortunately, we are not able to prove that the functions we obtain by descent are the theta functions, but 
we do show that in each case they generate
a representation that has non-zero projection to the space of theta.  Saying more would requite the analysis of archimedean
integrals (where the cover is trivial), which we do not undertake.  However,
we conjecture (Conjecture~\ref{conjecture1}) that in each case one may construct the
theta function from such a descent.  Our approach also suggests
that each of the conjectures of Patterson and Chinta-Friedberg-Hoffstein
is the first in a series of relations between various Fourier coefficients of automorphic
forms on different covering groups, obtained from a series of descent constructions on higher rank groups.

The key to this work is to adapt descent constructions to covering groups.
Descent is a method studied in Ginzburg, Rallis and Soudry \cite{G-R-S} which may be described as follows. Let $G$ denote an algebraic group, and let $\Theta$ denote a small representation of $G({\mathbb A})$,
where $\mathbb{A}$ denotes the ring of adeles of $F$.  In our context, a small representation may be defined as a representation which is not generic. In most examples, the representation $\Theta$ is a residue of an Eisenstein series, induced from a certain automorphic representation $\tau$ defined on a certain Levi subgroup of a parabolic subgroup of $G$. If needed, we write $\Theta_\tau$ to emphasize the dependence on $\tau$.

Let ${\mathcal O}$ denote a unipotent orbit of the group $G$. As explained in Ginzburg \cite{G1}, one can associate
to this unipotent orbit a unipotent subgroup $U$ of $G$, and a character $\psi_U$ of $U(F)\backslash U({\mathbb A})$. Assume that the stabilizer of this character inside  a suitable parabolic subgroup of $G$ is the reductive group $H$.
With this data one can construct the space of functions
\begin{equation}\label{intro1}
f(h)=\int\limits_{U(F)\backslash U({\mathbb A})}\theta_\tau(uh)\,\psi_U(u)\,du
\end{equation}
where $\theta_\tau\in\Theta_\tau$.
Denoting by $\sigma$ the representation of $H({\mathbb A})$ which is generated by all such $f(h)$, the main goal of the descent method is to understand the relation between the two representations $\tau$ and $\sigma$.
Another class of descent integrals may also be formed incorporating Fourier-Jacobi
coefficients; an example is given in (\ref{general1}) below.

So far all known examples of descent constructions
have been in the case where $\Theta_\tau$ is a representation on a linear group or on the double cover of a reductive group.
An important application has been when the correspondence between $\tau$ and $\sigma$ is functorial. This allows one to give examples of the Langlands correspondence or its inverse. Another application arises when both representations $\tau$ and $\sigma$ are generic, in which case the descent construction may be used to derive branching rules.  See \cite{G2} for such examples.

In this paper we initiate the study of descent constructions using small representations attached to arbitrary covering groups.  
In the metaplectic context, where little is known,
it may seem surprising that by using a representation on a larger
group one may be able to establish information about the resulting descent representation on a smaller group.  The key idea is to
choose a representation $\tau$ about which one already has some information.  
A source of such $\tau$ is the theorem of Kazhdan and Patterson \cite{K-P} that the
theta functions on certain $(n+1)$-fold covers of $GL_n$
have unique Whittaker models.  Thus theta functions on all such covers are more tractable and are good candidates for $\tau$.

With the applications to the Patterson and CFH Conjectures in mind, here
we concentrate on two specific cases. In the first case, we use a small representation defined on the 3-fold cover of $Sp_4$.  In this construction $\tau$ is the theta representation of the 3-fold cover of $GL_2$
and the resulting $\sigma$ is a representation on the 6-fold cover of $SL_2$. In the second case the small representation is defined on the 4-fold cover of $SO_7$, $\tau$ is defined on
a 4-fold cover of $GL_3$ and $\sigma$ on a 4-fold cover of $SO_4$.    To carry out the construction, we must
analyze the residue $\Theta_\tau$ of the resulting Eisenstein series.  We do so through unipotent orbits.

The unipotent
orbits are a partially ordered set, and a unipotent orbit ${\mathcal O}$ is {\sl attached} to
a representation if all Fourier coefficients of the form
(\ref{intro1}) for larger or incomparable orbits vanish identically and some coefficient
for ${\mathcal O}$ is nonzero. For an arbitrary metaplectic
covering of the group $GL_r$, the
determination of the unipotent orbits attached to a theta
representation has only just been accomplished (Cai, \cite{Cai}).
 For the double cover of odd orthogonal groups, this determination is the main step
in our paper \cite{B-F-G} with D.\ Bump.
Here, we show that in the $Sp_4$ case the representation $\Theta_\tau$ is attached to the unipotent orbit $(2^2)$, and in the $SO_7$ case it is attached to the orbit $(3^21)$.   Here we have identified unipotent orbits with partitions in the
standard way, based on the Jordan decomposition.  See Collingwood-McGovern \cite{C-M}, Section 5.1.

The key is that in both cases the descent integral allows us to relate the finite part of the Whittaker coefficients of the representation $\sigma$ to those of $\tau$.  
One may then use Hecke theory to simplify the relations, and they then
give two identities involving the Fourier coefficients of $\sigma$ and certain Gauss sums. This allows us to establish our main results, Theorems~\ref{maincfh} and \ref{mainso7}, which show that the $L^2$ functions we construct by
descent  have Fourier coefficients that satisfy the same periodicity properties as those established by
Kazhdan and Patterson for the theta functions, and moreover satisfy the same
relations as conjectured for the theta functions in the conjectures of CFH and Patterson.
We establish these results for all primes sufficiently close to $1$ at an (undetermined but finite) set of bad places.
The restriction to such primes is required since the descent constructions result in
local integrals at the ramified places 
which we do not have a precise way to compute.  Our results are consistent with but independent from Conjecture~\ref{conjecture1}.

The framework provided here suggests a more general series of conjectural relations between certain Fourier
coefficients of higher theta functions and Gauss sums. We outline the construction in one specific case.
Let $\tau$ denote the theta representation on a (certain) $(2n+1)$-fold cover of $GL_{2n}$.
Using this representation construct a representation $\Theta_\tau^{(2n+1)}$ which is defined on the $(2n+1)$-fold
cover of $Sp_{4n}$ as a residue of a certain Eisenstein series.
Let $\sigma$ denote the representation of the $(4n+2)$-fold cover of $Sp_{2n}$ generated by the space of functions
\begin{equation}\label{general1}
f(h)=\int\limits_{U(F)\backslash U({\mathbb A})}\widetilde{\theta}(l(u)h)\,\theta_\tau(uh)\,\psi_U(u)\,du.
\end{equation}
Here $\widetilde{\theta}$ is a function in the space of the minimal representation of the double cover of $Sp_{2n}$, and
$\theta_\tau$ is a function in the space of $\Theta_\tau^{(2n+1)}$. Also, $U$ is a certain unipotent subgroup of $Sp_{4n}$, $\psi_U$ is a character of $U(F)\backslash U({\mathbb A})$, and $l$ is a certain homomorphism from $U$ onto the Heisenberg group with $2n+1$ variables.  (In our case $\psi_U$ will be trivial but in higher rank cases it would appear.)

Then we expect that the representation $\Theta_\tau^{(2n+1)}$ is attached to the unipotent orbit $((2n)^2)$.
Once this is established, it should be possible to prove an identity similar to the CFH identity,
relating certain values of a Fourier coefficient of $\sigma$, a representation on the $(4n+2)$-fold
cover of $Sp_{2n}$, with certain $(2n+1)$-th order Gauss sums.  The case $n=1$ is treated here.
One should also be able to carry out similar constructions for orthogonal groups which generalize the construction
for $SO_7$ given here, and work with even orthogonal groups as well.

We thank James Arthur, Jim Cogdell, Erez Lapid, Colette M\oe glin and David Soudry for helpful conversations.

\section{Notations and Preliminaries}

\subsection{Basic Notations}\label{basic-notations}
Let $F$ denote a number field and let ${\mathbb{A}}$ denote its ring of adeles.
Since we work with covering groups, we will always require that $F$
contain certain groups of roots of unity, specified below.
If $\nu$ is place of $F$, we write $F_\nu$ for the completion of $F$ at $\nu$, and if $\nu$
is a finite place we write $O_\nu$ for the ring of local integers.
Let $\psi$ be a nontrivial character of $F\backslash {\mathbb{A}}$.
Let $G$ be an algebraic group, $V$ be a unipotent subgroup of $G$,
and $\psi_V$ be a given character of $V(F)\backslash V(\mathbb{A})$.
Let $\pi$ denote an automorphic representation of the group $G({\mathbb{A}})$
or of a cover of this group; for such covers the group $V(\mathbb{A})$ embeds
canonically via the trivial section \cite{M-W}.  For a given vector $\varphi$
in the space of $\pi$, denote
$$\varphi^{V,\psi_V}(h)=\int\limits_{V(F)\backslash V({\mathbb{A}})}
\varphi_\pi(vh)\,\psi_V(v)\,dv.$$

We will make use of the symplectic and orthogonal groups.
Let $J_n$ denote the $n\times n$ matrix with ones on the anti-diagonal and zeros elsewhere.
In terms of matrices, we will
represent the group $Sp_{2n}$ as the invertible linear transformations
that preserve the symplectic form given by $\begin{pmatrix}
&J_n\\ -J_n&\end{pmatrix}$, and the orthogonal group $SO_n$ as the invertible linear
transformations that preserve the symmetric bilinear form given by
$J_n$. All roots are defined with respect to the standard maximal
torus of diagonal matrices and positive roots are those corresponding
to the standard upper triangular unipotent subgroup.
We shall denote by $w_i$ the simple reflections of
the Weyl group. Order the positive roots so that in $Sp_4$, $w_1$ is the simple reflection
corresponding to the short root, and in the $SO_7$, $w_1$ and $w_2$
are the simple reflections which generate the Weyl group of $GL_3$
and $w_1$ and $w_3$, the third simple reflection of $SO_7$,
commute. We shall write $w[i_1i_2\ldots i_l]$ for
$w_{i_1}w_{i_2}\ldots w_{i_l}$.  If we work with matrices, we always represent $w_i$ by embedding
the matrix $\left(\begin{smallmatrix}0&1\\-1&0\end{smallmatrix}\right)$ at the corresponding root.

Let $\widetilde{\theta}_\phi^\psi$ denote the theta function defined on the
double cover of $SL_2({\mathbb{A}})$. Here $\phi$ is a Schwartz function of ${\mathbb{A}}$.
See for example \cite{G-R-S}, pg.\ 8.

Let $G$ be a reductive algebraic group. We denote by $G^{(n)}(\mathbb{A})$ an $n$-fold cover of the group
$G(\mathbb{A})$ as in Matsumoto \cite{M} (if $G$ is split and connected);
more generally see Brylinski-Deligne \cite{B-D}.  Such a cover is a group extension of $G(\mathbb{A})$ by
the group of $n$-th roots of unity $\mu_n$ and is defined
via a two-cocycle $\sigma$ which is constructed using local Hilbert symbols.
This cover exists when $F$ contains
enough roots of unity.
We recall that $G(F)$ embeds discretely in $G^{(n)}(\mathbb{A})$.
Similarly if $F_v$ is local we write $G^{(n)}(F_v)$ for an $n$-fold cover of $G(F_v)$.

Specifically, if $G=SL_r$, by $SL_r^{(n)}(F_v)$ we shall mean the $n$-fold cover of $SL_r$
which is described in detail in Kazhdan-Patterson \cite{K-P},
Section 0. This cover is determined by a two-cocycle $\sigma_n$
on $SL_r(F_v)$ which satisfies the property
\begin{equation}\label{cocycle-n}
\sigma_n(\text{diag}(h_1,\dots,h_r),\text{diag}(k_1,\dots,k_r))=\prod_{i<j}(h_i,k_j)_n,
\end{equation}
where $(~,~)_n$ is the $n$-th order local Hilbert symbol.  (We will drop the subscript $n$ when it is clear from context.)
This requires a choice of embedding of $\mu_n$ into $\mathbb{C}^\times$, which
we fix but do not incorporate into the notation.
We will also be concerned with the $n$-fold covers of $GL_2(\mathbb{A})$.
For the general linear group, the basic cocycle $\sigma_n$ satisfying \eqref{cocycle-n} may be twisted
to obtain a new cocycle
$\sigma_{n,c}(g_1,g_2)=\sigma_n(g_1,g_2)\,(\det g_1,\det g_2)^c$
with  $c\in \mathbb{Z}/n\mathbb{Z}$ .  See Kazhdan-Patterson \cite{K-P}, pg.\ 41.
Thus there is more than one cover, and in fact
the representation theory of different covers is not the same (\cite{K-P}, Cor.\ I.3.6).

In this paper we will be concerned with the
$3$-fold cover of $Sp_4(\mathbb{A})$ and the 4-fold covers of $SO_7(\mathbb{A})$, $SO_5(\mathbb{A})$
and $SO_4(\mathbb{A})$.
These covers are obtained via restriction from covers of the special linear group.
However, in forming the $3$-fold cover of $Sp_4$ it is convenient to use the $2$-cocycle
$$\sigma_{Sp_4}(g_1,g_2)=\sigma_3(wg_1w^{-1},wg_2w^{-1}),\qquad w=\begin{pmatrix}I_2&\\&J_2\end{pmatrix}.$$
In working with this group we require that $F$ contain
a full set of third roots of unity.  To form the 4-fold cover of $SO_7$ we
restrict the 8-fold cover of $SL_7(\mathbb{A})$ to $SO_7(\mathbb{A})$:
$$\sigma_{SO_7}(g_1,g_2)=\sigma_8(w'g_1w'^{-1},w'g_2w'^{-1}),\qquad
w'=\begin{pmatrix}I_4&\\&J_3\end{pmatrix}.$$Accordingly in this case
we shall assume that $F$ contains the eighth roots of unity. It is proper to regard this as a 4-fold rather
than 8-fold cover as the 4-th power of the cocycle $\sigma_{SO_7}$ is almost trivial.  See \cite{B-F-G}, Sect.~2.
The 4-fold cover of $SO_5$ is the restriction of this cover to $SO_5$ embedded in $SO_7$ via
\begin{equation}\label{so5-to-so7}
g\mapsto \begin{pmatrix}1&&\\&g&\\&&1\end{pmatrix}.
\end{equation}
The 4-fold cover of $SO_4$ is also obtained
from that for $SO_7$ via an embedding and will be described later.

We remark here that we could use $GSpin_7$ in place of $SO_7$.  The arguments below
concerning the constant terms and the various Fourier coefficients go through with
only minor technical changes.  The application to theta functions would be the same.

When $n$ is clear, we denote by $e_{i,j}$ the $n\times n$ matrix whose
$(i,j)$ entry is one and whose other entries are zero. We will also denote $e_{i,j}'=
e_{i,j}-e_{n+1-j,n+1-i}$.
Throughout this paper, when not specified,
matrices in $G(\mathbb{A})$ or $G(F_v)$ are embedded
in their covers by means of the trivial section $\mathbf{s}(g)=(g,1)$,
with one exception.  Over a local field $F_v$ such that $|n|_v=1$, for any $r\geq2$ the
compact subgroup $SL_r(O_v)$ embeds canonically in $SL_r^{(n)}$ (see \cite{K-P}, pgs.\ 43-44), 
$g\mapsto (g,\kappa(g))$.
This then restricts to an embedding of $G(O_v)$ into $G^{(n)}$
for each of the groups described above.  Because it plays little role in our calculations, we shall
use this embedding of $G(O_v)$ into the metaplectic group 
without introducing additional notation.  (We will not use $\kappa$ except
one place in Section~\ref{section6} below.)  At the finite places where $|n|_v\neq1$
a similar statement holds provided $SL_r(O_v)$ is replaced by a principal congruence subgroup.
Moreover, we shall always suppose that this subgroup is chosen
so that the diagonal elements in it are $n$-th powers in $F_v^\times$; for such elements $\kappa$ is $1$.
Lastly, we recall that the subgroup of upper triangular unipotents embeds in each cover
by the trivial section $\mathbf{s}$.

\subsection{Root Exchange}\label{exchange}

In the following sections during the computations, we will carry out
several Fourier expansions. One type of expansion will repeat
itself several times, and therefore it is convenient to state it in
generality. We shall refer to this process as {\sl root exchange}. We
recall this process briefly now;  a
detailed description is given in Ginzburg \cite{G}.
In our context, even though the functions we work
with are defined on covering groups, the cocycle arising from
the covering does not contribute any nontrivial factors
during the process of root exchange,
and so the description in \cite{G} applies.

Let $G$ be the adelic points of a split algebraic group or a cover of such.  Let
$\alpha$ and $\beta$ be two roots (not necessarily positive) and let $x_\alpha$ and $x_\beta$
be the standard one-parameter unipotents attached to $\alpha$ and $\beta$.
Let $U(\mathbb{A})$ be a unipotent subgroup that is normalized by $x_\alpha(t)$ and $x_\beta(t)$
for all $t\in \mathbb{A}$.  Let $f$ be an automorphic function.
We will be concerned with the possible vanishing of the integral
\begin{equation}\label{exch1}
\int\limits_{(F\backslash {\mathbb{A}})^2}\int\limits_{U(F)\backslash
U({\mathbb{A}})}f(ux_\alpha(m)x_\beta(r))\,\psi(m)\,du\,dm\,dr.
\end{equation}

To study this, consider the
following integral as a function of $g$:
\begin{equation}\label{exch2}\notag
L(g)=\int\limits_{F\backslash {\mathbb{A}}}\int\limits_{U(F)\backslash
U({\mathbb{A}})}f(ux_\alpha(m)g)\,\psi(m)\,du\,dm.
\end{equation}
We make the following hypothesis:  Suppose that $\gamma$ is
any root, positive or negative, which satisfies the
following commutation relations:  if $m,l,t\in \mathbb{A}$ then there
exist $u',u''\in U(\mathbb{A})$ such that
$$[x_\beta(l),x_\alpha(m)]=u';\ \ \ \ \ \ [x_\beta(l),x_\gamma(t)]=x_\alpha(l
t)u'';\ \ \ \ \ \ [x_\alpha(m),u'']=1.$$
Suppose that the function $L$ is left invariant under $x_\gamma(\delta)$ for
all $\delta\in F$, i.e.\ $L(x_\gamma(\delta)g)=L(g)$ for all $g\in G$.

In this case, we can expand the integral \eqref{exch1} along $x_\gamma(t)$
where $t\in F\backslash {\mathbb{A}}$. It is equal to
\begin{equation}\label{exch3}\notag
\sum_{\delta\in F}\int\limits_{(F\backslash {\mathbb
A})^3}\int\limits_{U(F)\backslash U({\mathbb
A})}f(ux_\alpha(m)x_\gamma(t)x_\beta(r))\,\psi(m+\delta t)\,dt\,du\,dm\,dr.
\end{equation}
Using the left-invariance properties of $f$ under the rational elements of $G$,
changing variables, and collapsing summation over $\delta$ with
integration over $r$, integral \eqref{exch1} is equal to
\begin{equation}\label{exch5}
\int\limits_{\mathbb{A}}\int\limits_{(F\backslash {\mathbb
A})^2}\int\limits_{U(F)\backslash U({\mathbb A})}
f(ux_\alpha(m)x_\gamma(t)x_\beta(r))\,\psi(m)\,dt\,du\,dm\,dr.
\end{equation}

Applying Lemma \ref{lem10} below, one deduces that the above
integral is zero for all choices of data if and only if the integral
\begin{equation}\label{exch6}
\int\limits_{(F\backslash {\mathbb{A}})^2}\int\limits_{U(F)\backslash
U({\mathbb{A}})}f(ux_\alpha(m)x_\gamma(t))\,\psi(m)\,dt\,du\,dm
\end{equation}
is zero for all choices of data. Hence, we deduce that the integral
\eqref{exch1} is zero for all choices of data if and only if
the  integral \eqref{exch6} is zero for all choices of
data. Referring to this process we will say that we exchange the
root $\beta$ by the root $\gamma$.
We will also call replacing the integral (\ref{exch1})
by the equal integral (\ref{exch5}) a root exchange.

We end this section with a general lemma that we will use several
times.  Let $\pi$ denote an automorphic representation of 
$G$, $U$ denote a unipotent subgroup of $G$ and $\psi_U$ 
denote a character of $U(F)\backslash U({\mathbb A})$. 
Given a function $f$ in $\pi$ denote by $f^{U,\psi_U}(g)$ the Fourier coefficient of $f$
determined by $U$ and $\psi_U$. Let $\{x(r)\mid r\in{\mathbb A}\}$ denote a one dimensional  unipotent subgroup of  $G$. Then
we have the following result.
\begin{lemma}\label{lem10}
Suppose that there is a one dimensional unipotent subgroup $\{y(m)\mid m\in{\mathbb A}\}$
of $G$ which satisfies the property
\begin{equation}\label{vanish10}
\int\limits_{\mathbb A}f^{U,\psi_U}(x(r)y(m))\,dr=
\int\limits_{\mathbb A}f^{U,\psi_U}(x(r))\,\psi(rm)\,dr
\end{equation}
for all $m\in {\mathbb A}$. Then the integral
\begin{equation}\label{vanish11}
\int\limits_{\mathbb A}f^{U,\psi_U}(x(r)g)\,dr
\end{equation}
is zero for all choices of data if and only if the Fourier coefficient $f^{U,\psi_U}(g)$ is
zero for all choices of data.
\end{lemma}
\begin{proof}
Assume that the integral \eqref{vanish11} is zero for all choices of data.
Let $\phi$ be any Schwartz function of ${\mathbb A}$. Then the integral
\begin{equation}\label{vanish12}
\int\limits_{{\mathbb A} ^2}f^{U,\psi_U}(x(r)y(m))\,\phi(m)\,dm\,dr\notag
\end{equation}
is zero for all choices of data. Using \eqref{vanish10}, we deduce that
\begin{equation}\label{vanish13}
\int\limits_{\mathbb A}f^{U,\psi_U}(x(r))\,\hat{\phi}(r)\,dr\notag
\end{equation}
is zero for all choices of data, where $\hat{\phi}$ is the Fourier 
transform of $\phi$. The Lemma follows.
\end{proof}

\section{Theta
Representations}\label{definitions}

\subsection{Definition and Basic Properties of the Theta Representation}\label{subsec31}
In this subsection we give the definition of the theta
representations for the groups $Sp_4^{(3)}({\mathbb{A}})$ and
$SO_7^{(4)}({\mathbb{A}})$. These are obtained globally
as residues of certain Eisenstein series.
The material in this subsection is an adaptation
to these groups of the construction of the theta representation for covers of the general
linear group in  \cite{K-P}. We also refer to \cite{B-F-G} where
a similar representation for the double cover of an odd orthogonal group was
constructed and studied, and \cite{F-G-S} where further work with
theta representations is carried out.  We also note that a theta function on the cubic cover
of $Sp_4$ was constructed for $F=\mathbb{Q}(\exp(2\pi i /3))$ by Proskurin \cite{Pro}.

Let $H$ be one of the groups $Sp_4$ and $SO_7$ and $n=3, 4$ resp.  Let $B$ denote its
standard Borel subgroup and write $B=TU$ where $T$ is the maximal
torus of $H$. Over a local field or over the
adeles, we denote the inverse of the groups $B$ and $T$ in
the covering group by
$\widetilde{B}$ and $\widetilde{T}$. Let $Z(\widetilde{T})$ denote
the center of $\widetilde{T}$,
and fix a maximal abelian
subgroup $\widetilde{A}$ of $\widetilde{T}$ containing $Z(\widetilde{T})$.
(The group $\widetilde{A}$ need not be the inverse image of a subgroup of $T$.)
It follows from \cite{K-P} that
representations of $\widetilde{B}$ are determined as follows. Fix a genuine
character $\chi$ of $Z(\widetilde{T})$, extend it in any way to a character of
$\widetilde{A}$, and then induce it to
$\widetilde{T}$. Then extend this representation trivially to $\widetilde{B}$.
Inducing the resulting representation of $\widetilde{B}$ to $H^{(n)}$ we obtain a
representation of this group which we denote by
$Ind_{\widetilde{B}}^{H^{(n)}}\chi$.  The notation is justified since
up to isomorphism this representation depends only on $\chi$
(and not on the choice of $\widetilde{A}$, in particular).
This may be carried out
over a local field or over the adeles.

In carrying out this process, it will be convenient to begin with a character $\chi$ of the torus
$T$.  We may restrict $\chi$ to the projection of $Z(\widetilde{T})$ to $T$
and then extend it to a genuine character of $Z(\widetilde{T})$ by settting
$\chi((t,\zeta))=\zeta\chi(t)$.
For this character we construct the induced representation as
explained above. Henceforth, when there is no confusion, we will
only specify the character $\chi$ on $T$.

The global theta representation will be defined as a residue of an
Eisenstein series induced from the Borel subgroup. Parameterize the
maximal torus of $Sp_4$ as
\begin{equation}\label{t-sp4}
t(a,b)=\text{diag}(a,b,b^{-1},a^{-1})
\end{equation}
 and that
of $SO_7$ as
\begin{equation}\label{t-so7}
t(a,b,c)=\text{diag}(a,b,c,1,c^{-1},b^{-1},a^{-1}).
\end{equation}
Unless mentioned otherwise, we will always use this parametrization.

We start with the case $H=Sp_4$. The group $Sp_4$ has two maximal parabolic subgroups. The first,
which we denote by $P$, has the group $GL_2$ as its Levi subgroup.
The second parabolic whose Levi subgroup is $GL_1\times SL_2$ we
shall denote by $Q$. A straightforward computation shows that when we
restrict the cocycle $\sigma_{Sp_4}$ from $Sp_4$ to the Levi part of $P$, that is to
$GL_2$, we obtain the complex conjugate of the $GL_2$ cocycle $\sigma_{3,1}$.
This will be important later, since the Whittaker coefficient of the
theta representation of the group $GL_2^{(3)}({\mathbb{A}})$
constructed using the cocycle $\sigma_{3,1}$ is not
unique. For the conditions of the uniqueness, see
\cite{K-P} Corollary I.3.6.

Let $\chi_{s_1,s_2}$ denote the character of $T$ defined by
$\chi_{s_1,s_2}=\delta_{B_{GL_2}}^{s_1}\delta_P^{s_2}$. In
coordinates we have
$\chi_{s_1,s_2}(t(a,b))=|a|^{s_1+3s_2}|b|^{-s_1+3s_2}$. Let
$E_B^{(3)}(h,s_1,s_2)$ denote the Eisenstein series of
$Sp_4^{(3)}({\mathbb{A}})$ associated with the induced representation
$Ind_{\widetilde{B}}^{Sp_4^{(3)}}\chi_{s_1,s_2}$.
(Here and below we often suppress the choice of section from the notation.)
Following Prop.~II.1.2 in \cite{K-P}, or Section 3 in
\cite{B-F-G} we deduce that the poles of this Eisenstein series are
determined by the poles of
$$\frac{\zeta(6s_1-3)\zeta(18s_2-9)\zeta(3s_1+9s_2-6)\zeta(-3s_1+9s_2-3)}
{\zeta(6s_1-2)\zeta(18s_2-8)\zeta(3s_1+9s_2-5)\zeta(-3s_1+9s_2-2)},$$
where $\zeta$ is the Dedekind zeta function of $F$.
This ratio has a multi-residue at $s_1=2/3$ and $s_2=2/3$.

\begin{definition}
The theta representation on $Sp_4^{(3)}(\mathbb{A})$ is
\begin{equation*}
\Theta_{Sp_4}^{(3)}=\text{res}_{s_2=2/3}\text{res}_{s_1=2/3}E_B^{(3)}(\cdot,s_1,s_2).
\end{equation*}
\end{definition}

Let $\Theta_{GL_2}^{(3)}$ be
the theta representation of the three-fold cover of $GL_2({\mathbb A})$ with $c=1$ as
defined in \cite{K-P}.
Let $E_{P}^{(3)}(h,s)$ denote the space of Eisenstein series of
$Sp_4^{(3)}({\mathbb{A}})$ associated with the induced representation
$Ind_{\widetilde{P}({\mathbb{A}})}^{Sp_4^{(3)}({\mathbb
A})}\overline{\Theta_{GL_2}^{(3)}}\delta_P^s$. Then it follows from the above, using
induction in stages, that the functions in
$E_{P}^{(3)}(h,s_2)$ are exactly the residues at $s_1=2/3$
of the functions in $E_B^{(3)}(h,s_1,s_2)$.
This means that we can obtain the representation $\Theta_{Sp_4}^{(3)}$ by
the residues of a space of Eisenstein series associated with a theta
representation of the Levi subgroup of a maximal parabolic. In fact
the same goes for the other maximal parabolic subgroup
$Q$  of $Sp_4$. Let $E_{Q}^{(3)}(h,s)$ denote the space of Eisenstein
series of $Sp_4^{(3)}({\mathbb{A}})$ associated with the induced
representation $Ind_{\widetilde{Q}({\mathbb{A}})}^{Sp_4^{(3)}({\bf
A})}\Theta_{SL_2}^{(3)}\delta_Q^s$. Here $\Theta_{SL_2}^{(3)}$ is
the theta representation of the three-fold cover of $SL_2$.
Then the representation $\Theta_{Sp_4}^{(3)}$
can be obtained as the residues at $s=2/3$ of the Eisenstein series in $E_{Q}^{(3)}(h,s)$.

A local constituent of this representation is constructed in a manner similar
to the two references above. Namely, let $F_\nu$ be a non-archimedean local field.
Consider the induced representation
$Ind_{\widetilde{B}(F_\nu)}^{Sp_4^{(3)}(F_\nu)}\chi_\Theta\delta_B^{1/2}$.
Here $\chi_\Theta(t)=|a|^{2/3}|b|^{1/3}$. Then one may define the
local representation $(\Theta_{Sp_4}^{(3)})_\nu$ as the image of
the intertwining operator from
$Ind_{\widetilde{B}(F_\nu)}^{Sp_4^{(3)}(F_\nu)}\chi_\Theta\delta_B^{1/2}$
to
$Ind_{\widetilde{B}(F_\nu)}^{Sp_4^{(3)}(F_\nu)}\chi_\Theta^{-1}\delta_B^{1/2}$.
Similarly for $\nu$ archimedean, so that $Sp_4^{(3)}(F_\nu)\cong Sp_4(\mathbb{C})\times \mu_3$,
there is a local theta representation obtained by smooth induction.
More precisely, the smooth induced module is irreducible in
this case, and is also the image of the intertwining
operator, as in \cite{K-P}, Theorem I.6.4 part (c).
Arguing as in the above two references one obtains that
$\Theta_{Sp_4}^{(3)}=\otimes_\nu (\Theta_{Sp_4}^{(3)})_\nu$. We omit
the details.

We return to the global situation. We will need to study the
constant terms of the representation $\Theta_{Sp_4}^{(3)}$ along the
unipotent radicals of the two maximal parabolic subgroups. We argue
in the same way as in \cite{B-F-G}, Prop.~3.4. Let $U(P)$
denote the unipotent radical of $P$. We consider the constant term
$$E_{P}^{(3),U(P)}(h,s)=\int\limits_{U(P)(F)\backslash U(P)({\mathbb
A})}E_P^{(3)}(uh,s)\,du.$$
The space of double cosets $P(F)\backslash
Sp_4(F)/P(F)$ contains three elements with representatives $e, w_2$
and $w[212]$. For $m\in GL_2$, let $i_P(m)$ denotes its image in the Levi
subgroup of $P$.  Then we obtain
$$E_{P}^{(3),U(P)}(i_P(m),s,f_s)=f_s(i_P(m))+E_{GL_2}(m,M_{w_2}f_s',s')+M_{w[212]}f_s(i_P(m)).$$
Here the Eisenstein series is formed with section $f_s\in Ind_{\widetilde{P}({\mathbb{A}})}^{Sp_4^{(3)}({\mathbb
A})}\overline{\Theta_{GL_2}^{(3)}}\delta_P^s$, the $M_w$ denote intertwining operators,
and $E_{GL_2}(m,M_{w_2}f_s',s')$
is the Eisenstein series on the 3-fold cover of $GL_2$
associated with the induction from the Borel with induction data
$\delta_B^{s'}\text{det}^{4/3}$ where $s'=3s-1$.  The specific section used is
$$f_s'(m)=\int\limits_{F\backslash {\mathbb{A}}}f_s\left (
\begin{pmatrix} 1&r&&\\ &1&&\\ &&1&-r\\ &&&1\end{pmatrix}i_P(m)\right
)dr$$ It is not hard to check that the first two terms are
holomorphic at $s=2/3$, the point where $E_{P}^{(3),U(P)}(h,s)$ has
a simple pole.

A similar argument applies to the second maximal parabolic $Q$. The space
of double cosets $Q(F)\backslash Sp_4(F)/Q(F)$ again consists of three elements, this time
represented by $e,w_1,w[121]$. Thus, when we
compute the constant term $E_{Q}^{(3),U(Q)}(i_Q(a,h),s)$ where $a\in
GL_1$, $h\in SL_2$, and $i_Q$ denotes the embedding into the Levi subgroup of $Q$, 
we get three terms.  Again, two of them will
contribute zero after taking the residue at $s=2/3$. We
summarize the above discussion in the following Proposition.  (We recall that by convention
that elements of $H(\mathbb{A})$ are embedded in the cover via the trivial section $\mathbf{s}$.)
\begin{proposition}\label{propconstant1}
\begin{enumerate}[(i)]
\item
Suppose that
$t(a,b)$ is in
$\widetilde{A}(\mathbb{A})$. Let $h=\text{diag}(a,b)\in
GL_2^{(3)}({\mathbb{A}})$.
Then for each vector $\theta_{Sp_4}^{(3)}\in \Theta_{Sp_4}^{(3)}$
there is a vector $\theta_{GL_2}^{(3)}\in \Theta_{GL_2}^{(3)}$ such
that
\begin{equation*}
\theta_{Sp_4}^{(3),U(P)}(t(a,b))=|\det h|\,
\overline{\theta}_{GL_2}^{(3)}(h).
\end{equation*}
\item
Suppose that $t(a,1)$ is in
$\widetilde{A}({\mathbb{A}})$.  Let $h\in SL_2^{(3)}({\mathbb{A}})$.
Then for each vector $\theta_{Sp_4}^{(3)}\in \Theta_{Sp_4}^{(3)}$ there is a vector
$\theta_{SL_2}^{(3)}\in\Theta_{SL_2}^{(3)}$ such that
\begin{equation*}
\theta_{Sp_4}^{(3),U(Q)}(t(a,1)i_Q(1,h))=|a|^{4/3}\,\theta_{Sp_4}^{(3),U(Q)}(i_Q(1,h))
=|a|^{4/3}\theta_{SL_2}^{(3)}(h).
\end{equation*}
\end{enumerate}
\end{proposition}

\begin{proof}
The Proposition is a standard consequence of the constant term calculations above.
For example to prove the first part we start with the identity
\begin{equation*}
\theta_{Sp_4}^{(3),U(P)}(t(a,b))=\text{res}_{s=2/3}M_{w[212]}f_s(t(a,b)).
\end{equation*}
Since the intertwining operator $M_{w[212]}$ maps
$Ind_{\widetilde{P}({\mathbb{A}})}^{Sp_4^{(3)}({\mathbb
A})}\overline{\Theta_{GL_2}^{(3)}}\delta_P^s$ onto
$Ind_{\widetilde{P}({\mathbb{A}})}^{Sp_4^{(3)}({\mathbb
A})}\overline{\Theta_{GL_2}^{(3)}}\delta_P^{1-s}$ it follows that at
$s=2/3$ we obtain the factor of $\delta_P^{1/3}(h)=|\text{det}\ h|$.

The second part follows in a similar way.
\end{proof}

The situation with the group $SO_7$ is similar. In fact, the
definition and the basic properties follow exactly as in
\cite{B-F-G}.  Start with the global induced
representation $Ind_{\widetilde{B}({\mathbb{A}})}^{SO_7^{(4)}({\mathbb
A})}\chi_{s_1,s_2,s_3}\delta_B^{1/2}$. Here $B$ is the Borel
subgroup of $SO_7$, and 
$\chi_{s_1,s_2,s_3}(t(a,b,c))=|a|^{s_1}|b|^{s_2}|c|^{s_3}$. Then form the space of
Eisenstein series $E_B^{(4)}(h,s_1,s_2,s_2)$ associated with this
induced representation. Then, as in \cite{B-F-G}, the poles of this
Eisenstein series are determined by the poles of
$$\frac{\prod_{1\le i<j\le
3}\zeta(4(s_i-s_j))\zeta(4(s_i+s_j))\prod_{1\le i\le 3}\zeta(4s_i)}
{\prod_{1\le i<j\le
3}\zeta(4(s_i-s_j)+1)\zeta(4(s_i+s_j)+1)\prod_{1\le i\le
3}\zeta(4s_i+1)}.$$ This quotient has a multiple residue at $s_1=3/4,
s_2=2/4$ and $s_3=1/4$. We denote by $\Theta_{SO_7}^{(4)}$ the
residue representation. Let
$\chi_\Theta(t(a,b,c))=|a|^{3/4}|b|^{2/4}|c|^{1/4}$. Then
$\Theta_{SO_7}^{(4)}$ is a sub-quotient of $Ind_{\widetilde{B}({\mathbb
A})}^{SO_7^{(4)}(\mathbb{A})}\chi_\Theta\delta_B^{1/2}$. It can also be
realized as a subrepresentation of $Ind_{\widetilde{B}({\mathbb
A})}^{SO_7^{(4)}({\mathbb{A}})}\chi_\Theta^{-1}\delta_B^{1/2}$.

The group $SO_7$ has three maximal parabolic subgroups. By analogy with
the $Sp_4$ case, let us now denote by
$P$ the maximal parabolic subgroup whose Levi factorization is given (up to isomorphism) 
by $GL_3\,U(P)$, by $Q$ the maximal parabolic with Levi factorization
$Q=(GL_2\times SO_3)U(Q)$ and by $R$ the maximal parabolic with factorization
$R=(GL_1\times SO_5)U(R)$, and by $i_P$ (resp.\ $i_Q$, $i_R$) the corresponding
embedding of $GL_3$ (resp.\ $GL_2\times SO_3$, $GL_1\times SO_5$) into 
the corresponding Levi subgroup in $SO_7$.  For $g_1,g_2\in GL_3$, the cocycle $\sigma_{SO_7}(i_P(g_1),i_P(g_2))$ is 
the standard cocycle $\sigma_4(g_1,g_2)$ for the 4-fold cover of $GL_3$, twisted by
the eighth-order Hilbert symbol $(\det g_1,\det g_2)_8^{-1}$.  It is not difficult to
see that the construction of the theta representation in \cite{K-P} goes through
without essential change in this case.  In particular, the Whittaker model for this representation
is not unique; the argument is exactly that given in Theorem I.3.5 there using $c=-1/2$,
and Corollary I.3.6 holds. We denote this representation $\Theta_{GL_3}^{(4)}$.

Using induction in stages we can realize the representation
$\Theta_{SO_7}^{(4)}$ as a residue of a space of Eisenstein series
associated with an induced representation from a maximal parabolic
subgroup. Starting with the group $P$, we define the space of Eisenstein
series $E_P^{(4)}(h,s)$ associated with $Ind_{\widetilde{P}({\mathbb
A})}^{SO_7^{(4)}({\mathbb{A}})}\Theta_{GL_3}^{(4)}\delta_P^s$. Then, it
follows similarly to the above that these Eisenstein series have
simple poles at $s=2/3$. The residue representation is
$\Theta_{SO_7}^{(4)}$. For the group $Q$ we have a similar result.
If we write the corresponding Eisenstein series as $E_Q^{(4)}(h,s)$,
then $\Theta_{SO_7}^{(4)}$ is the space of residues of these Eisenstein series at
the point $\delta_Q^{21/32}$.  Finally, for the Eisenstein series
$E_R^{(4)}(h,s)$ the representation $\Theta_{SO_7}^{(4)}$ is  the space of
residues at the point $\delta_R^{13/20}$.

Using this, we may establish a result analogous to Prop.~\ref{propconstant1}.
The third part of the following
Proposition  follows by an argument similar to \cite{B-F-G}, Prop.~3.4.
We omit the details.  

\begin{proposition}\label{propconstant2}

\begin{enumerate}[(i)]
\item
Suppose that
$t(a,b,c)$ is in  $\widetilde{A}({\mathbb{A}})$.
Let $h=\text{diag}(a,b,c)\in GL_3^{(4)}({\mathbb{A}})$. Then for each vector
$\theta_{SO_7}^{(4)}\in \Theta_{SO_7}^{(4)}$ there is a vector
$\theta_{GL_3}^{(4)}\in \Theta_{GL_3}^{(4)}$ such that
\begin{equation*}
\theta_{SO_7}^{(4),U(P)}(t(a,b,c))=|\det h|\,\theta_{GL_3}^{(4)}(h).
\end{equation*}
\item
Suppose that $t(a,a,1)$ is
in $\widetilde{A}({\mathbb{A}})$.  Let $h\in SO_7^{(4)}(\mathbb{A})$. Then, for all
$\theta_{SO_7}^{(4)}\in \Theta_{SO_7}^{(4)}$ we have
\begin{equation*}
\theta_{SO_7}^{(4),U(Q)}(t(a,a,1)\,h)=|a|^{9/8}\,\theta_{SO_7}^{(4),U(Q)}(h).
\end{equation*}
\item
Suppose that $t(a,1,1)$  is in
$\widetilde{A}({\mathbb{A}})$. 
Let $h\in SO_5^{(4)}({\mathbb A})$.  Then for each vector 
$\theta_{SO_7}^{(4)}\in \Theta_{SO_7}^{(4)}$ there is a vector $\theta_{SO_5}^{(4)}\in\Theta_{SO_5}^{(4)}$ such that
\begin{equation*}
\theta_{SO_7}^{(4),U(R)}(t(a,1,1)\,i_R(1,h))=|a|^{7/4}\,\theta_{SO_7}^{(4),U(R)}(i_R(1,h))=|a|^{7/4}\,\theta_{SO_5}^{(4)}(h).
\end{equation*}
\end{enumerate}

\end{proposition}
We remark that the power $|\det h|$ in the first part and the powers
of $|a|$ in the second and third parts are just the values of $\delta^{1-s_0}$
where $s_0$ is the point where the corresponding Eisenstein series
has a pole.

There are also local analogues of 
Props.~\ref{propconstant1} and \ref{propconstant2}.  We shall need only the analogue
of Prop.~\ref{propconstant2}, (iii), in the sequel
so limit ourselves to this case.
For each nonarchimedean place $\nu$ let $(\Theta_{SO_7}^{(4)})_\nu$ denote the local theta representation.
This representation is 
the image of the intertwining operator from
$Ind_{\widetilde{B}(F_\nu)}^{{SO_7}^{(4)}(F_\nu)}\chi_\Theta\delta_B^{1/2}$
to
$Ind_{\widetilde{B}(F_\nu)}^{{SO_7}^{(4)}(F_\nu)}\chi_\Theta^{-1}\delta_B^{1/2}$.
Let $J_{U(R)}$ be the Jacquet module with respect to the unipotent radical of the parabolic subgroup $R$.  Recall that this is a functor
from the category of smooth representations of finite length of $SO_7^{(4)}(F_\nu)$ to the category of
smooth representations of finite length of $\widetilde{M_R}(F_\nu)$, the full inverse image of the Levi subgroup
of $R(F_\nu)$ in the covering
group $SO_7^{(4)}(F_\nu)$.  This inverse image may be identified
with the quotient 
of the direct product $GL^{(4)}_1(F_\nu)\times SO^{(4)}_5(F_\nu)$ obtained by identifying the central subgroups.
Then, as a representation of $SO_5^{(4)}(F_\nu)$ the Jacquet module 
$J_{U(R)}(\Theta_{SO_7}^{(4)}(F_\nu))$ is the theta representation $\Theta_{SO_5}^{(4)}$. The proof of this 
is the same as \cite{B-F-G}, Theorem 2.3, which treats the same statement for the double (rather than four-fold)
cover of an odd orthogonal group.

Let $n$ be odd. We end this section with the construction of induced
representations on the even-degree cover $SL_2^{(2n)}$ of $SL_2$,
following \cite{B-F-H}. Suppose that the embeddings of $\mu_n$ and
$\mu_{2n}$ into $\mathbb{C}$ are chosen so that the Hilbert
symbols satisfy $(a,b)_{2n}=(a,b)_2\,(a,b)_n$. Then the cocycle
$\sigma_{2n}$ on $SL_2$ satisfies $\sigma_{2n}=\sigma_2\sigma_n$.
Let $B=TU$ now denote the Borel subgroup of $SL_2$ and $A$ denote
any maximal subgroup of $T$ which satisfies $\sigma_n(a_1,a_2)=1$
for all $a_1,a_2\in {A}$. Then $\widetilde{A}$ is a maximal abelian
subgroup of $\widetilde{T}$.  Let $\chi$ denote a character of the center of $\widetilde{T}$.
We extend it to a character $\widetilde{\chi}$ of $\widetilde{A}$:
if $b=\left( \left(\begin{smallmatrix} a&\\
&a^{-1}\end{smallmatrix}\right),\zeta\right )$ is in
$\widetilde{A}$, define $\widetilde{\chi}(b)=\chi(a)\gamma(a)\zeta$.
Here $\zeta\in\mu_{2n}$, and $\gamma$ is the Weil factor
attached to the nontrivial additive character $\psi$. As in \cite{B-F-H} one
may check that $\widetilde{\chi}$ is indeed a genuine character of
$\widetilde{A}$. Inducing from $\widetilde{A}$ to $\widetilde{T}$
and extending in the usual way to $\tilde{B}$ we can then form the
induced representation
$Ind_{\widetilde{B}}^{SL_2^{(2n)}}\widetilde{\chi}$.

The particular case of interest for us below is the case of the theta representation.
Let $n=3$, and let $\chi_s=\delta_B^{s}$. We then form the global
induced  representation $Ind_{B^{(6)}({\mathbb A})}^{SL_2^{(6)}({\mathbb
A})}\widetilde{\chi}_s$. Attached to this representation is the Eisenstein series $E^{(6)}(g,s)$ on
$SL_2^{(6)}({\mathbb A})$. As in the previous cases it is not hard to
check that this Eisenstein series has a simple pole at $s=7/12$;
we denote the residue representation by $\Theta_{SL_2}^{(6)}$. See also
Conjecture~\ref{conjecture1} below.

\subsection{ Fourier Coefficients of the Theta Representation}

In this subsection we study the Fourier coefficient attached to the theta representations
of the various groups under consideration here.  Our goal here is to determine the unipotent orbit attached to
each such representation in the sense of \cite{B-F-G}, Section 4,  where
such a determination is given for the theta function on the double cover of an odd orthogonal group.
 We start with

\begin{lemma}\label{lem1}
The theta representations $\Theta_{Sp_4}^{(3)}$, $\Theta_{SO_5}^{(4)}$, and
$\Theta_{SO_7}^{(4)}$ are not generic.
\end{lemma}

This Lemma is a special case of a more general result of Friedberg, Goldberg and
Szpruch \cite{F-G-S}.  The proof is similar to that for covers of $GL_r$ \cite{K-P}.
For the convenience of the reader we give a brief sketch here.

\begin{proof}
It suffices to show that there is no local Whittaker functional.  As in subsection~\ref{subsec31}, let
$n=3$ or $n=4$ for the groups under consideration, as shown.  Let $\nu$ be a place
such that $|n|_\nu=1$, and
let $F_\nu$ be the corresponding local field.  Fix one of the above covering groups over $F_\nu$.
We use the same notation for the principal series as in subsection~\ref{subsec31}.
Let $\omega$ be a genuine quasicharacter of $Z(\widetilde{T})$, extend to
a character $\omega'$ of a maximal abelian subgroup $\widetilde{A}$ containing $Z(\widetilde{T})$
as in \cite{K-P},  Sect.~I.1, and form
the principal series, denoted in this proof by $V(\omega')$.
The dimension of the space of Whittaker functionals on $V(\omega')$ is easily seen
to be $[\tilde{T}:\widetilde{A}]$, and one can write down a basis by means of the regularizations $\lambda_\eta$ of standard
Whittaker-type integrals
$$\int_{U} f(\eta w_0^{-1}u)\,\bar\psi(u)\,du$$
as $\eta$ runs over the quotient $\widetilde{A}\backslash \tilde{T}$.  Here $\psi$ is a nondegenerate
character of the maximal unipotent subgroup $U$. One then considers the functionals
$f\to\lambda_\eta\circ I_w(f)$ with $\lambda_\eta$ the corresponding Whittaker integral for $V({\omega'}^w)$.
These may be expressed (on a Zariski dense set) as linear combination $\sum_{\eta'} \tau_w(\eta,\eta')\lambda_{\eta'}$
(compare \cite{K-P}, pp.\ 75-76).  The coefficients $\tau_w$ are seen to satisfy certain relations as in \cite{K-P}, Lemma I.3.3.

Suppose now that the quasicharacter
$\omega$ is exceptional.  In our context, this means that
$\omega_\alpha^{n(\alpha)}=|~|_\nu$ for all positive
roots $\alpha$, where $\omega_\alpha$ is the composition of $\omega$ with the canonical embedding
attached to $\alpha$, and where in the symplectic case $n(\alpha)=3$ and in the orthogonal group case $n(\alpha)=4$ if $\alpha$
is long and $2$ if $\alpha$ is short.  (Compare \cite{K-P}, pg.\ 71.)
These are exactly the characters considered in subsection~\ref{subsec31} above.
Then the local exceptional representation $V_0(\omega')$ is the image of $V(\omega')$ under the intertwining operator $I_{w_0}$ attached
to the long Weyl element $w_0$ on $V(\omega')$.
(Since $\omega$ is regular the representations $V(\omega')$, $V_0(\omega')$ are independent of the choice of the extension $\omega'$, as in \cite{K-P}, Prop.~I.2.2.)
Now a Whittaker functional on $V_0(\omega')$ pulls back to one on $V(\omega')$.
But every Whittaker functional on $V(\omega')$ may be written as
$\sum_{\eta\in \widetilde{A}\backslash \tilde{T}} c(\eta) \lambda_\eta$ where $c:\tilde{T}\to\mathbb{C}$ satifies
$c(a t)=\delta_B(a)^{-1/2} c(t)$ for $a\in \tilde{A}$.
The Whittaker functionals in $V_0(\omega')$ are those pulling back to ones on $V(\omega')$
that are zero on the kernel of $I_{w_0}$.  This gives a system of equations for the $c(\eta)$, which may be expressed
using the coefficients $\tau_w(\eta,\eta')$ (see \cite{K-P}, Cor.\ I.3.4.).  This allows one to count the dimension of the space of
Whittaker functionals on $V_0(\omega')$.  As in \cite{K-P}, Theorem I.3.5, this dimension
is seen to be at most the number of free orbits of the Weyl group $W$ acting on the quotient lattice $\mathbb{Z}^r/n\mathbb{Z}^r$,
where $r$ is the rank of the group.
In the cases at hand it is easy to check that there are no free orbits, and so no Whittaker functionals.
\end{proof}

\begin{proposition}\label{prop3}
The unipotent orbit attached to $\Theta_{Sp_4}^{(3)}$ is $(2^2)$ and the unipotent orbit attached to $\Theta_{SO_7}^{(4)}$ is $(3^21)$.
\end{proposition}

\begin{proof}
We start with the $Sp_4$ case. It follows from Lemma~\ref{lem1} that $\Theta_{Sp_4}^{(3)}$ is not generic. On the other hand,
it follows as in \cite{G1}, Theorem 3.1,  that  a representation on
$Sp_4^{(3)}({\mathbb A})$ which is not generic must have a nonzero Fourier coefficient associated with the unipotent orbit $(2^2)$.

Next consider the $SO_7$ case.
We need to prove that the representation $\Theta_{SO_7}^{(4)}$ has no nonzero
Fourier  coefficient corresponding to the unipotent orbits $(7)$ and $(51^2)$, and
has a nonzero Fourier coefficient corresponding to the orbit $(3^21)$.
Lemma~\ref{lem1} above states that $\Theta_{SO_7}^{(4)}$ is not generic;
this is equivalent to the vanishing of the Fourier coefficients for the orbit $(7)$.

Consider the Fourier coefficients corresponding to the orbit $(51^2)$. They are given as follows.
Let $U'$ denote the standard unipotent radical of the parabolic subgroup of $SO_7$
whose Levi part is $GL_1^2\times SO_3$. Then the Fourier coefficients are given by
\begin{equation*}
\theta_{SO_7}^{U',\psi_{U',\alpha}}(h)=\int\limits_{U'(F)\backslash U'({\mathbb{A}})}
\theta(uh)\,\psi_{U',\alpha}(u)\,du.
\end{equation*}
Here $\theta$ is a vector in the space $\Theta_{SO_7}^{(4)}$, and
$\psi_{U',\alpha}$ is defined as follows. Let
$\alpha\in F^\times$ and given $u=(u_{i,j})$ we define $\psi_{U',\alpha}(u)=
\psi(u_{1,2}+u_{2,3}-\frac{1}{2}\alpha u_{2,5})$. Changing $\alpha$ by a square in $F^\times$
gives the  same function for a different $h$, so we analyze $\alpha$ modulo $(F^\times)^2$.

Suppose first that $\alpha=1$. Then conjugating
by a suitable matrix in $SO_3(F)$, the above Fourier coefficient is  zero for all
choices of data if and only if the integral
\begin{equation}\label{four5112}
\int\limits_{U'(F)\backslash U'({\mathbb{A}})}
\theta(uh)\,\psi_{U'}(u)\,du
\end{equation}
is zero for all choices of data. Here $\psi_{U'}(u)=\psi(u_{1,2}
+u_{2,4})$.

Conjugating from left to right by $w[12]$, we then perform  root
exchanges as explained in subsection \ref{exchange}. Let $\{z(m_1,m_2)\}$ denote the unipotent subgroup of $SO_7$ generated by all 
matrices $z(m_1,m_2)=I_7+m_1e'_{1,3}+m_2e'_{1,4}-\frac{m_2^2}{2}e_{1,7}$ and let $\{y(r_1,r_2)\}$ denote the unipotent subgroup of $SO_7$ generated by all  matrices $y(r_1,r_2)=I_7+r_1e_{2,1}'+r_2e_{3,1}'$.
Then using root exchange we deduce that
the  integral \eqref{four5112} is equal to
$$\int\limits_{\mathbb{A}^2}\int\limits_{U_1(F)\backslash U_1({\mathbb{A}})}\theta
(u_1y(r_1,r_2)w[12])\,\psi_{U_1}(u_1)\,du_1\,dr_1\,dr_2.$$
Here $U_1$ is the subgroup of $U$, the
maximal unipotent subgroup of $SO_7$, such that if $u=(u_{i,j})\in U_1$, then $u_{1,2}=0$, and
the character $\psi_{U_1}$ of $U_1$ is given by $\psi_{U_1}(u)=\psi(u_{2,3}+u_{3,4})$.
We note that the above root exchange is to be carried out
in stages. In other words, we first perform the root exchange of
the group $\{z(0,m_2)\}$ with the group $\{y(0,r_2)\}$ and then of $\{z(m_1,0)\}$ with $\{y(r_1,0)\}$. 

Expand the above integral along the
subgroup $y(r)=I_7+re_{1,2}'$. There are two orbits to consider. The nontrivial orbits
contribute zero to the expansion. Indeed, when this is the case we obtain, as inner
integration the Whittaker coefficient of $\theta$, which vanishes by Lemma~\ref{lem1}.
Thus, the above integral is equal to
$$\int\limits_{\mathbb{A}^2}\int\limits_{U(F)\backslash U({\mathbb{A}})}\theta
(uy(r_1,r_2)w[12])\psi_{U}(u)\,du\,dr_1\,dr_2$$ where $\psi_U(u)=\psi(u_{2,3}+u_{3,4})$.
Applying Lemma \ref{lem10},
we conclude that the integral \eqref{four5112} is zero for all choices of data if and only if the integral
\begin{equation}\label{integral-section3}
\int\limits_{U(F)\backslash U({\mathbb{A}})}\theta(u)\,\psi_U(u)\,du
\end{equation}
is zero for all choices of data. 

Write $U=U''U(R)$ where $U(R)$ is the unipotent radical of the
maximal parabolic subgroup $R$ of $SO_7$ whose Levi part is $GL_1\times SO_5$,
and $U''$ is the maximal
unipotent radical of $SO_5$ embedded in $SO_7$ as in \eqref{so5-to-so7}.
From the definition of $\psi_U$ it is trivial on
$U(R)$. Writing \eqref{integral-section3}
as an iterated integral and computing first the constant term along $U(R)$,
we see using Prop.~\ref{propconstant2}, part (iii), 
that we obtain the Whittaker coefficient of the theta representation of $SO_5^{(4)}$.
By Lemma~\ref{lem1}, the
representation $\Theta_{SO_5}^{(4)}$ has no nonzero Whittaker
coefficients, and hence the integral \eqref{integral-section3} is zero for all choices of data.

Next consider the case when $\alpha$ is not a square in $F^\times$. In this case, we choose a local
unramified place $\nu$ such that $\alpha$ is a square in $F_\nu^\times$. It is enough to prove that the 
twisted Jacquet module
of  the local  representation  $(\Theta_{SO_7}^{(4)})_\nu$
corresponding to the unipotent subgroup $U'$ and character $\psi_{U'}$ is zero. 
We do so by employing the Lemma established in \cite{G-R-S1},
Section 2.2, which gives an isomorphism of twisted Jacquet modules in certain circumstances. 
The Lemma is stated there for the symplectic group, but the statement and proof for the orthogonal group are
identical. In \cite{G-R-S} Section 7.1, a similar result is also stated and
proved in the global context in complete generality.

To explain the argument, let $U'_1=w[12]U'w[12]^{-1}$, and for 
$u=(u_{i,j})\in U'_1$ let
$\psi_{U'_1}(u)=\psi(u_{2,3}+u_{3,4})$.
Then our goal is to show that the twisted Jacquet module $J_{U'_1,\psi_{U'_1}}((\Theta_{SO_7}^{(4)})_\nu)$ is zero.  
In the notation of the Lemma in \cite{G-R-S1}, let $C$ denote the subgroup of $U_1$ of matrices $u=(u_{i,j})\in U_1$ 
such that $u_{1,3}=u_{1,4}=0$. Also, let $X=\{z(0,m_2)\}$ and $Y=\{y(0,r_2)\}$. 
Then it is easy to see that the conditions of the Lemma are satisfied. Here $D_0=CX$. 
Notice that $D_0$ is a subgroup of $U_1$, and we may consider the character $\psi_{U_1}$ as a character of $D_0$. 
We repeat this process again with $C=D_0$,  $X=\{z(m_1,0)\}$ and  $Y=\{y(r_1,0)\}$. 
Notice that this time $U_1=CX$. Doing so, it follows from the Lemma  that 
$$J_{U'_1,\psi_{U'_1}}((\Theta_{SO_7}^{(4)})_\nu)\cong J_{U_1,\psi_{U_1}}((\Theta_{SO_7}^{(4)})_\nu).$$
Hence it is enough to prove that the latter Jacquet module is zero. To establish this, it is enough to prove that the twisted Jacquet 
module $J_{U,\psi_{U,\beta}}((\Theta_{SO_7}^{(4)})_\nu)$ is zero for all $\beta\in F_\nu$. Here, 
$\psi_{U,\beta}$ is the character $\psi_{U,\beta}(u)=\psi(\beta u_{1,2}+u_{2,3}+
u_{3,4})$ of $U$. If $\beta\ne 0$, then since we proved in Lemma~\ref{lem1}  that the representation $\Theta_{SO_7}^{(4)}$
has no local Whittaker functional,
the corresponding Jacquet module is zero. If $\beta=0$ we write 
$U=U''U(R)$ as defined above. We apply the argument immediately after Proposition \ref{propconstant2}  to deduce 
that $J_{U(R)}((\Theta_{SO_7}^{(4)})_\nu)$, as a representation of $SO_5^{(4)}(F_\nu)$,
is the theta representation on this group.  Since by Lemma~\ref{lem1} the
representation $\Theta_{SO_5}^{(4)}$ also has no local Whittaker functional,
we deduce that the Jacquet module $J_{U'',\psi_{U_1}}((\Theta_{SO_5}^{(4)})_\nu)$ is zero. 
Hence the desired vanishing holds when $\alpha$ is not a square as well.
(In fact, this argument gives a local proof of vanishing when $\alpha=1$ too. The reader may find it helpful to compare
the global and local arguments.)

To complete the proof of the Proposition, we need to show that $\Theta_{SO_7}^{(4)}$ has a nonzero
Fourier coefficient corresponding to the unipotent orbit $(3^21)$. Consider the
Fourier coefficient
\begin{equation}\label{four3310}
\int\limits_{U(Q)(F)\backslash U(Q)({\mathbb{A}})}
\theta(uh)\,\psi_{Q}(u)\,du
\end{equation}
where $U(Q)$ is the unipotent radical of the
standard maximal parabolic subgroup $Q$ whose Levi part is $GL_2\times SO_3$. The character
$\psi_Q$ is defined as $\psi_Q(u)=\psi(u_{1,3}+u_{2,5})$. It is not hard to show
that the stabilizer of $\psi_Q$ inside $GL_2\times SO_3$ is a one-dimensional torus, and
this Fourier coefficient corresponds to the orbit $(3^21)$.  
To show that these Fourier coefficients are not zero
for all choices of data,
conjugate integral \eqref{four3310} by the Weyl element $w[32]$. We obtain the integral
\begin{equation}\label{four3311}
\int\limits_{(F\backslash {\mathbb{A}})^2}\int\limits_{V(F)\backslash V({\mathbb{A}})}
\theta(vz(r_1,r_2)w[32])\,\psi_V(v)\,dv\,dr_1\,dr_2.
\end{equation}
Here $V$ is the subgroup of $U'$ such that if $v=(v_{i,j})\in V$, then $v_{1,5}
=v_{2,4}=v_{2,5}=0$. Thus its dimension is five. The character $\psi_V$ is defined as
$\psi_V(v)=\psi(v_{1,2}+v_{2,3})$. Finally, we define $z(r_1,r_2)=I_7+r_1e_{5,2}'+r_2e_{4,3}'
-\frac{1}{2}r_2^2e_{5,3}$.

Expand the integral \eqref{four3311} along the two-dimensional unipotent
subgroup given by
$y(l_1,l_2)=I_7+l_1e_{1,5}'+l_2e_{2,4}'-\frac{1}{2}l_2^2e_{2,6}.$
Performing root exchanges as explained in subsection \ref{exchange}
we replace $z(r_1,r_2)$ by $y(l_1,l_2)$. Thus, the integral
\eqref{four3311} is equal to
\begin{equation}\label{four3312}
\int\limits_{{\mathbb{A}}^2}\int\limits_{V_1(F)\backslash V_1({\mathbb{A}})}
\theta(v_1z(r_1,r_2)w[32])\,\psi_{V_1}(v_1)\,dv_1\,dr_1\,dr_2.
\end{equation}
Here $V_1$ is the subgroup of $U'$ such that if $v_1=(v_1(i,j))\in V_1$ then $v_1(2,5)=0$.

Next,
we expand the integral \eqref{four3312} along the subgroup $x_1(l)=I_7+le_{2,5}'$. The contribution
from the nontrivial orbits is zero. This follows from the fact that each such
Fourier coefficient corresponds to the unipotent orbit $(51^2)$, so was shown to
vanish above. Finally, we expand along the group $x_2(l)=I_7+le_{3,4}'-
\frac{1}{2}l^2e_{3,5}$. As above, the contribution from the nontrivial orbit is zero, since
$\Theta_{SO_7}^{(4)}$ is not generic. Thus, the above integral is equal to
\begin{equation}\label{four3313}
\int\limits_{{\mathbb{A}}^2}\int\limits_{U(F)\backslash U({\mathbb{A}})}
\theta(uz(r_1,r_2)w[32])\,\psi_{U,1}(u)\,du\,dr_1\,dr_2 .
\end{equation}
Here $U$ is the maximal unipotent subgroup of $SO_7$ and $\psi_{U,1}(u)=\psi(
u_{1,2}+u_{2,3})$.
Applying Lemma \ref{lem10}, we deduce that this integral is not zero for some choice of data
if and only if the integral
\begin{equation*}
\int\limits_{U(F)\backslash U({\mathbb{A}})}
\theta(u)\,\psi_{U,1}(u)\,du 
\end{equation*}
is not zero for some choice of data. Using Prop.~\ref{propconstant2} part (i), 
we see that this integral is not zero for some choice of data if
and only if the representation $\Theta_{GL_3}^{(4)}$ is not the zero representation.  Since this is so (see \cite{K-P}),
we deduce the non-vanishing of the integral \eqref{four3310} for some choice of data, as claimed.
\end{proof}

\section{ The Descent Integrals}

In this Section we define the descent integrals and study the constant terms and
the Whittaker coefficients of the resulting functions.  In carrying out these computations,
we work with matrices that are embedded in their covers by means of the trivial section,
but we continue to suppress this section from the notation.  Of course, when roots of unity
arise from the cocycle (for example, in \eqref{non-trivial-cocycle} below) we shall 
include them.

\subsection{The Descent From $Sp_4$}
Starting with the representation $\Theta_{Sp_4}^{(3)}$, we consider the function
\begin{equation*}
f_{SL_2}^{(6)}(g)=\int\limits_{(F\backslash {\mathbb
A})^3}\widetilde{\theta}_\phi^\psi
((x,y,z)g)\,\theta_{Sp_4}^{(3)}(i(x,y,z)i(g))\,dx\,dy\,dz.
\end{equation*}
where $i$ denotes the embeddings into $Sp_4$ given by 
\begin{equation*}  
i(x,y,z)=\begin{pmatrix} 1&x&y&z\\ &1&&y\\ &&1&-x\\ &&&1\end{pmatrix}\quad\text{and}\quad
i(g)=\begin{pmatrix} 1&&\\ &g&\\ &&1\end{pmatrix}.
\end{equation*}
The function $\theta_{Sp_4}^{(3)}$ is a vector in the space of
$\Theta_{Sp_4}^{(3)}$, and for convenience we shall denote it by
$\theta$.
Notice that $f_{SL_2}^{(6)}(g)$ is a function defined on the group $SL_2^{(6)}({\mathbb{A}})$
which is left-invariant under the discrete subgroup $SL_2(F)$.  
Let $\sigma_{SL_2}^{(6)}$ denote the representation of $SL_2^{(6)}({\mathbb{A}})$
generated by all the above functions.

\subsubsection{\bf The Constant Term} We start by computing the constant term of
$f_{SL_2}^{(6)}(g)$ with $g=t(a):=\text{diag}(a,a^{-1})$, $a\in {\mathbb A}^\times$.
Unfolding the theta function we obtain
\begin{equation}\label{sp44}
\int\limits_{F\backslash {\mathbb{A}}}f_{SL_2}^{(6)}\left (\begin{pmatrix} 1&r\\
&1\end{pmatrix}\begin{pmatrix} a&\\ &a^{-1}\end{pmatrix}\right )dr=
\end{equation}
$$\int\limits_{{\mathbb{A}}}\int\limits_{(F\backslash {\mathbb{A}})^3}\omega_\psi
((0,y,z)(x,0,0)n(r)t(a))\,\phi(0)\,\theta(i(0,y,z)i(x,0,0)i(n(r))t(1,a))\,dy\,dz\,dr\,dx.$$
Here $n(r)=\left(\begin{smallmatrix} 1&r\\
&1\end{smallmatrix}\right)$, the matrix $t(1,a)$ is defined in (\ref{t-sp4}),
and $\omega_\psi$ is the Weil representation with character $\psi$
realized in the Schr\"odinger model (see, for example,
\cite{G-R-S} Sect.~1.2). Using the properties of the
Weil representation, and conjugating by $w_1$ in the function
$\theta$, we deduce that integral \eqref{sp44} is equal to
\begin{multline*}
|a|^{-1/2}\gamma(a)\int\limits_{{\mathbb{A}}}\int\limits_{(F\backslash {\mathbb
A})^3}\phi(x)\theta\left (\begin{pmatrix} 1&&y&r\\ &1&z&y\\ &&1&\\
&&&1\end{pmatrix}\begin{pmatrix} a&&&\\ &1&&\\ &&1&\\
&&&a^{-1}\end{pmatrix}w_1\begin{pmatrix} 1&x&&\\ &1&&\\ &&&1&-x\\
&&&&1\end{pmatrix} \right )\times\\ \psi(z)\,dy\,dz\,dr\,dx.
\end{multline*}
The factor $|a|^{-1/2}\gamma(a)$
arises from the action of the Weil representation together with a
change of variables in $x$. Also, $\gamma(a)$ is the Weil factor (a root of unity)
associated to $\psi$.

Next we expand the above integral along the group $I_4+le_{1,2}'$.
Since $\Theta_{Sp_4}^{(3)}$ is not generic (see Lemma~\ref{lem1}),
the contribution from the nontrivial characters is zero. Thus
integral \eqref{sp44} is equal to
\begin{equation}\label{sp45}
|a|^{-1/2}\gamma(a)\int\limits_{{\mathbb{A}}}\phi(x)\theta^{U,\psi_U}\left (
\begin{pmatrix} a&&&\\ &1&&\\ &&1&\\
&&&a^{-1}\end{pmatrix}w_1\begin{pmatrix} 1&x&&\\ &1&&\\ &&&1&-x\\
&&&&1\end{pmatrix} \right )dx
\end{equation}
Here $U$ is the standard maximal unipotent subgroup of $Sp_4$ and
$\psi_U$ is defined by $\psi_U(u)=\psi(u_{2,3})$.
Note that for any nontrivial $\psi$,
\begin{equation}\label{whsl21}
\int\limits_{F\backslash {\mathbb{A}}}\theta_{SL_2}^{(3)}\begin{pmatrix} 1&x\\ &1\end{pmatrix}\,
\psi(x)\,dx
\end{equation}
is not zero for some choice of data.  This may be seen using the Hecke relations and the
Fourier expansion.  
Consequently the expression
\eqref{sp45} is nonzero for some choice of data. Indeed, plugging in $a=1$,  if it were  zero for all choices of
data, then it follows from the second part of Prop.~\ref{propconstant1}  that the integral \eqref{whsl21} is zero for all
choices of data.

\subsubsection{\bf The Whittaker Coefficient} In this subsection we
compute a Whittaker coefficient of the function $f_{SL_2}^{(6)}$.
More precisely, we compute the integral
\begin{equation}\label{sp4wh1}
\text{Wh}_2^{(6)}(g)=
\int\limits_{F\backslash {\mathbb{A}}}f_{SL_2}^{(6)}\left (\begin{pmatrix} 1&r\\
&1\end{pmatrix}g\right
)\psi(-\frac{1}{4}r)\,dr
\end{equation}
Here $g\in SL_2^{(6)}({\mathbb A})$.
Beginning as in the computation of the constant term, we see that the
integral \eqref{sp4wh1} is equal to
\begin{equation}\label{whittaker1}
\int\limits_{{\mathbb{A}}}\int\limits_{(F\backslash {\mathbb
A})^3}\omega_\psi(g)\phi(x)\,\theta\left (\begin{pmatrix} 1&&y&z\\ &1&r&y\\ &&1&\\
&&&1\end{pmatrix}\begin{pmatrix} 1&x&&\\ &1&&\\ &&&1&-x\\
&&&&1\end{pmatrix}i(g) \right ) \psi(z-\frac{1}{4}r)\,\,dy\,dz\,dr\,dx.
\end{equation}
Let $\gamma_0=\begin{pmatrix} 1&\\1&1\end{pmatrix}
\begin{pmatrix} 1&-\frac{1}{2}\\&1\end{pmatrix}=\begin{pmatrix} 1&-\frac{1}{2}\\ 1&\frac{1}{2}\end{pmatrix}$,
$\gamma_1=\begin{pmatrix} \gamma_0&\\
&\gamma_0^*\end{pmatrix}$,
where $\gamma_0^*$ is chosen so that the matrix $\gamma_1$ is symplectic. Since $\gamma_1\in Sp_4(F)$, we
can conjugate it from left to right. After a change of
variables, the character $\psi$ is changed, and we obtain
$$\int\limits_{{\mathbb{A}}}\int\limits_{(F\backslash {\mathbb
A})^3}\omega_\psi(g)\phi(x)\,\theta\left (\begin{pmatrix} 1&&y&z\\ &1&r&y\\ &&1&\\
&&&1\end{pmatrix}\gamma_1 \begin{pmatrix} 1&x&&\\ &1&&\\ &&&1&-x\\
&&&&1\end{pmatrix}i(g) \right ) \psi(y)\,dy\,dz\,dr\,dx.$$

Next we perform a root
exchange as explained in subsection \ref{exchange}. This allows us to replace the
one-dimensional unipotent subgroup $z_1(r)=I_4+re_{2,3}$ by
$m(l)=I_4+le_{1,2}'$.  Then we conjugate by $w_2$ from left to
right, and we obtain the integral
\begin{multline*}
\int\limits_{{\mathbb{A}}^2}\int\limits_{(F\backslash {\mathbb
A})^3}\omega_\psi(g)\phi(x)\,\theta\left (\begin{pmatrix} 1&y&l&z\\ &1&&l\\ &&1&-y\\
&&&1\end{pmatrix}w_2z_1(r)\gamma_1 \begin{pmatrix} 1&x&&\\ &1&&\\ &&&1&-x\\
&&&&1\end{pmatrix} i(g)\right )\\ \times \psi(y)\,dy\,dz\,dr\,dx\,dl.
\end{multline*}
Then we
expand the above integral along the unipotent group $I_4+me_{2,3}$.
The contribution from the nontrivial characters is zero since, by Lemma~\ref{lem1},
$\Theta_{Sp_4}^{(3)}$ is not generic. We thus obtain
$$\int\limits_{{\mathbb{A}}^2}\omega_\psi(g) \phi(x)\theta^{U,\psi_{U,1}}
(w_2z_1(r)\gamma_1m(x)i(g))\,dr\,dx.$$
Here $\psi_{U,1}$ is defined by $\psi_{U,1}(u)=\psi(u_{1,2})$.  

We next apply the Bruhat decomposition
$\gamma_0=\left (\begin{smallmatrix} 1&1\\ &1\end{smallmatrix}\right )\left (\begin{smallmatrix} &-1\\ 1&\end{smallmatrix}\right )\left (\begin{smallmatrix} 1&1/2\\ &1\end{smallmatrix}\right )$.
Making use of the corresponding decomposition of $\gamma_0^*$ in $Sp_4(F)$,  the matrix $\left(\begin{smallmatrix} &-1\\
1&\end{smallmatrix}\right)$ is replaced by the simple reflection $w_1$.  Thus we obtain
$$\int\limits_{{\mathbb{A}}^2}\omega_\psi(g) \phi(x)\theta^{U,\psi_{U,1}}
(w_2z_1(r)m(1)w_1m(x+\tfrac{1}{2})i(g))\,dr\,dx.$$

Finally, we conjugate $m(1)$ to the left, and $w_2$ and $w_1$ to the right,
and we deduce that integral \eqref{sp4wh1} is equal to the integral
\begin{equation}\label{sp4wh2}
\int\limits_{{\mathbb
A}^2}\omega_\psi(g)\phi(x)\,
\theta^{U,\psi_{U,1}}(z(r)w_2w_1m(x+\tfrac{1}{2})i(g))\,\psi(r)\,dr\,dx
\end{equation}
where $z(r)=I_4+re_{3,2}$.  In the above integral,
$\psi(r)$ arises from the conjugation of $m(1)$ across
$z_1(r)$ and a change of variables in $U$.

As a first consequence, this computation of the Whittaker coefficient
gives the following non-vanishing result.

\begin{lemma}\label{nonzerosp4}
The representation $\sigma_{SL_2}^{(6)}$ is not zero, and moreover
for some choice of data the Whittaker coefficient \eqref{sp4wh1} is not zero.
\end{lemma}

\begin{proof}
 The representation $\sigma_{SL_2}^{(6)}$ is not zero since we have shown that some function in
this space has a nonzero constant term.
 To prove the nonvanishing of some Whittaker coefficient, it is enough to prove that the integral \eqref{sp4wh2} is not zero for some choice of data. Let $g=e$, and assume
instead that this integral is zero for all choices of data. Since $\phi$ is an arbitrary Schwartz function, this implies that the integral
$$\int\limits_{\mathbb A}\theta^{U,\psi_{U,1}}(z(r))\,\psi(r)\,dr$$
is zero for all choices of data.  Using the subgroup $\{I_4+m(e_{1,3}+e_{2,4})\mid m\in\mathbb{A}\}$ 
and applying Lemma \ref{lem10}, we deduce that  $\theta^{U,\psi_{U,1}}$ is zero for all choices of data. Since the representation $\Theta_{GL_2}^{(3)}$ is generic, we obtain a contradiction to Prop.~\ref{propconstant1}, part (i).
\end{proof}

\subsection{ The Descent From $SO_7$} We start by describing the descent
integral. Recall that $U(R)$ denotes the standard unipotent radical of the
maximal parabolic subgroup of $SO_7$ whose Levi part is $GL_1\times
SO_5$. We define a character of this group by
$\psi_{U(R)}(u)=\psi(u_{1,4})$. It is not hard to check that the
stabilizer of $\psi_{U(R)}$ inside $SO_5$ contains the split orthogonal
group $SO_4$. For $g$ in this $SO_4$, we let $i(g)=i_R(1,g)\in SO_7$.
(We use $i(g)$ for the embedding in each descent integral to emphasize the
parallel structure; there will be no ambiguity as the two descents are treated separately.)
We define the descent integral by
\begin{equation}\label{so71}
f_{SO_4}^{(4)}(g)=\int\limits_{U(R)(F)\backslash U(R)({\mathbb
A})}\theta_{SO_7}^{(4)}(u\, i(g))\,\psi_{U(R)}(u)\,du.
\end{equation}
Here $\theta_{SO_7}^{(4)}$ is a vector in the space of
$\Theta_{SO_7}^{(4)}$, and we write $\theta$ if there is no
confusion. As in the previous case, the functions $f_{SO_4}^{(4)}(g)$ are invariant
under $SO_4(F)$.  We denote by
$\sigma_{SO_4}^{(4)}$ the representation of
$SO_4^{(4)}({\mathbb{A}})$ generated by the space of functions given by the
integral \eqref{so71}.

\subsubsection{\bf The Constant Term} We start with the constant
term of the descent. In other words, we first compute
\begin{equation}\label{so72}
\int\limits_{(F\backslash {\mathbb{A}})^2}f_{SO_4}^{(4)}\left
(\begin{pmatrix} 1&r_1&r_2&-r_1r_2\\ &1&&-r_2\\ &&1&-r_1\\
&&&1\end{pmatrix}
\begin{pmatrix} a&&&\\ &b&&\\ &&b^{-1}&\\ &&&a^{-1}\end{pmatrix}
\right ) dr_1\,dr_2,
\end{equation}
where $a,b\in \mathbb{A}^\times$.
Here the embeddings of the matrices in the argument of $f_{SO_4}^{(4)}$ into $SO_7$ via the map $i$ are given by
\begin{equation}\label{so73}
(I_7+r_1e_{2,3}')(I_7+r_2e_{2,5}')\quad\text{and}\quad\text{diag}(1,a,b,1,b^{-1},a^{-1},1).
\end{equation}
To compute the constant term we first start by exchanging the
unipotent subgroup $y_1(x_1)=I_7+x_1e_{1,2}'$ with
$I_7+le_{2,3}'-\frac{1}{2}l^2e_{2,6}$. Then conjugating from the
left by $w_1$, the integral \eqref{so72} is equal to
$$\int\limits_{\mathbb{A}}\int\limits_{U(Q)(F)\backslash U(Q)({\mathbb{A}})}
\theta(v w_1y_1(x_1)t(1,a,b))\,\psi_{U(Q)}(v)\,dv\,dx_1.$$ Here $t(1,a,b)$
is the torus element in \eqref{so73} (cf.\ (\ref{t-so7})),and we recall that $U(Q)$ is the standard
unipotent radical of the maximal parabolic subgroup of $SO_7$ whose
Levi part is $GL_2\times SO_3$. The character $\psi_{U(Q)}$ is
defined by $\psi_{U(Q)}(v)=\psi(v_{2,4})$.

Expand this integral
along the unipotent group $I_7+re_{1,2}'$. The nontrivial characters
contribute zero to the expansion, since each of them is a Fourier
coefficient associated with the unipotent orbit $(51^2)$. By Prop.~\ref{prop3}
this is zero. Thus we are left with the contribution
from the constant term. Next we exchange the unipotent subgroup
$y_2(x_2)=I_7+x_2e_{2,3}'$ with
$I_7+le_{3,4}'-\frac{1}{2}l^2e_{3,5}$. Conjugating by $w_2$ from
left to right the above integral is equal to
$$\int\limits_{{\mathbb{A}}^2}\int\limits_{V_2(F)\backslash V_2({\mathbb{A}})}
\theta(v_2w_2y_2(x_2)w_1y_1(x_1)t(1,a,b))\,\psi_{V_2}(v_2)\,dv_2\,dx_1\,dx_2.$$
Here $V_2$ is the subgroup of the maximal unipotent subgroup $U$ of
$SO_7$ such that $v=(v_{i,j})\in V_2$ if and only if $v_{2,3}=0$.
Expanding the above integral along the group $I_7+le_{2,3}'$, it again
follows from Prop.~\ref{prop3} that the nontrivial characters
contribute zero to the expansion. Thus we deduce that integral
\eqref{so72} is equal to
\begin{equation}\label{so74}
\int\limits_{{\mathbb
A}^2}\theta^{U,\psi_{U,2}}(w_2y_2(x_2)w_1y_1(x_1)\,t(1,a,b))\,dx_1\,dx_2.
\end{equation}
Here
$\psi_{U,2}$ is defined by $\psi_{U,2}(u)=\psi(u_{3,4})$.

\subsubsection{\bf The Whittaker Coefficient}
In this subsection we compute a certain Whittaker coefficient of the
function $f_{SO_4}^{(4)}(g)$. More precisely, we compute the
integral
\begin{equation}\label{whso720}
\int\limits_{(F\backslash {\mathbb{A}})^2}f_{SO_4}^{(4)}\left
(\begin{pmatrix} 1&r_1&r_2&-r_1r_2\\ &1&&-r_2\\ &&1&-r_1\\
&&&1\end{pmatrix}
\begin{pmatrix} a&&&\\ &1&&\\ &&1&\\ &&&a^{-1}\end{pmatrix}
\right )\,\psi(r_1-\frac{1}{2}r_2) \, dr_1\, dr_2.
\end{equation}
The embeddings of these matrices in $SO_7$ are given in \eqref{so73},
and we start with the same exchange of unipotent subgroups as in the
computation of the constant term. We obtain the integral
$$\int\limits_{\mathbb{A}}\int\limits_{U(Q)(F)\backslash U(Q)({\mathbb{A}})}
\theta(uy_1(x_1)t(1,a,1))\,\psi'_{U(Q)}(u)\,du\,dx_1$$ where $\psi'_{U(Q)}$ is defined
as $\psi'_{U(Q)}(u)=\psi(u_{1,4}+u_{2,3}-\frac{1}{2}u_{2,5})$.
Let
$$\gamma_1=\begin{pmatrix} 1&\frac{1}{2}\\
1&-\frac{1}{2}\end{pmatrix},\ \ \ \
\gamma_2=k(1)\begin{pmatrix}&&-1\\ &-1&\\ -1&&\end{pmatrix}k(1)\ \ \text{with}
\ \  k(1)=\begin{pmatrix} 1&1&-\frac{1}{2}\\ &1&-1\\
&&1\end{pmatrix}$$ and denote
$\gamma=\text{diag}(\gamma_1,\gamma_2,\gamma_1^*)\in SO_7(F)$.
Then $\gamma$ normalizes the subgroup $U(Q)$.
Conjugating this element from left to right
changes the character $\psi'_{U(Q)}$ to a new character whose
value on $u\in U(Q)$ is $\psi(u_{1,3}+u_{2,5})$. Thus the integral
over $U(Q)$ with this character is exactly the integral
\eqref{four3310}. Therefore we may perform the same Fourier
expansion that leads from integral \eqref{four3310} to integral
\eqref{four3313}.  We conclude that the above integral is equal to
\begin{equation}\label{whso72}
\int\limits_{{\mathbb{A}}^3} \theta^{U,\psi_{U,1}}(z(r_1,r_2)w[32]\gamma
y_1(x_1)t(1,a,1))\,dr_1\,dr_2\,dx_1.\notag
\end{equation}

The $\gamma_2$ part of $\gamma$ commutes with $y_1(x_1)t(a,1)$,
hence we can conjugate it to the right, and ignore it by means of a
change of the vector $\theta$. Performing the Bruhat decomposition
of $\gamma_1$ and changing variables in $x_1$, the above integral becomes
\begin{equation}\label{whso73}\notag
\int\limits_{{\mathbb{A}}^3} \theta^{U,\psi_{U,1}}(z(r_1,r_2)w[32]y_1(1)w_1
y_1(x_1)t(1,a,1))\,dr_1\,dr_2\,dx_1.
\end{equation}
Conjugating the matrix $y_1(1)$ to the left and the matrix $w[321]$ to
the right and changing vectors, we obtain the integral
\begin{equation}\label{whso74}
|a|^{-1}\int\limits_{{\mathbb{A}}^3} \theta^{U,\psi_{U,1}}(z(r_1,r_2)\,t(a,1,1)\,
y_2(x_1))\,\psi(r_1)\,dr_1\,dr_2\,dx_1.
\end{equation}
Here $y_2(x_1)=I_7+x_1e_{5,1}'$,  the factor of $|a|^{-1}$ is
obtained from the change of variables in $x_1$, and $\psi(r_1)$
is obtained from the conjugation of $y_1(1)$ to the left and a change of
variables.

Similarly to Lemma~\ref{nonzerosp4} we have

\begin{lemma}\label{nonzeroso7}
The representation $\sigma_{SO_4}^{(4)}$ is not zero,
and moreover for some choice of data the Whittaker coefficient \eqref{whso720}
is not zero.
\end{lemma}

\subsection{Additional Expressions for the Whittaker Coefficients of the Descent}
To conclude this section, 
we work further with the expressions given above.
Let $M$ denote any of the algebraic groups considered here and 
$M^{(n)}({\mathbb{A}})$ denote a covering group for some number $n$. 
Let $\pi$ denote an
irreducible representation of $M^{(n)}({\mathbb{A}})$, and write
$\pi=\pi_{\text{inf}}\otimes\pi_{\text{fin}}$. Here
$\pi_{\text{inf}}$ is the product of all local representation
$\pi_\nu$ such that $\nu$ is an infinite place, and
$\pi_{\text{fin}}$ is the product over all finite places in $F$.

Assume that $\pi$ is generic, i.e.\ that the integral
$$W_\pi(m)=\int\limits_{U_M(F)\backslash U_M({\mathbb
A})}\varphi_\pi(um)\,\psi_{U_M}(u)\,du$$ is not zero for some choice of
data. Here $U_M$ is the maximal unipotent radical of $M$,
$\psi_{U_M}$ is a generic character of $U_M(F)\backslash U_M({\mathbb
A})$, and the function $\varphi_\pi$ is a vector in the space of
$\pi$. Choose a function $\varphi_\pi$ which corresponds to a
factorizable vector in $\pi$. Then
\begin{lemma}\label{lem3}
With the above notations we have
$W_\pi(m)=W_{\pi_{\text{inf}}}(m_{\text{inf}})W_{\pi_{\text{fin}}}(m_{\text{fin}})$.
Here $W_{\pi_{\text{inf}}}$ is a Whittaker functional defined over
all infinite places. Also, $W_{\pi_{\text{fin}}}$ is a Whittaker
functional on the representation $\otimes_{\nu\in\Phi}\pi_\nu$,
where $\Phi$ is the set of all finite places in $F$.

\end{lemma}

\begin{proof}
Since all infinite places are complex, the cover is trivial in
those places. Therefore, the space of Whittaker functionals defined on these
local representations is one-dimensional, and hence can be separated
from the other places. (For a  similar argument see the Basic Lemma
in \cite{PS-R}, pg.~117.)
\end{proof} 

In view of Lemma~\ref{lem3}, we may study the finite part of the Whittaker coefficients of the descent integrals.
Choose factorizable
vectors in the corresponding representations, and let $\phi_\Phi=\prod_{\nu\in\Phi}\phi_\nu$ where $\phi_\nu$ is a
Schwartz function of $F_\nu$ at the place $\nu$ such that for almost
all $\nu$ it is the unramified function.
Let ${\mathbb A}_\Phi=\prod_{\nu\in\Phi}F_\nu$,
 $|a|_\Phi=\prod_{\nu\in\Phi}|a_\nu|_\nu$,  and
$\phi_\Phi=\prod_{\nu\in\Phi}\phi_\nu$.  Then we have

\begin{lemma}\label{newlemma1} Let $g=\text{diag}(a,a^{-1})$.  Then the Whittaker coefficient \eqref{sp4wh1} has finite 
part given by
\begin{equation}\label{factor1}
|a|_\Phi^{-1/2}\gamma(a)\int\limits_{{\mathbb
A}_\Phi^2}\phi_\Phi(x)\,L_\Phi(z(r)t(a,1)w_2w_1m(x+\frac{a}{2}))\,\psi(r)\,dr\,dx,
\end{equation}
where the function $L_\Phi$ is described below and is computed on a certain maximal abelian subgroup of
$\widetilde{T}({\mathbb A}_\Phi)$ by  \eqref{factor2} below.
\end{lemma}

\begin{proof}
To prove this, we obtain an integral for the finite part
of the Whittaker coefficient directly from Eqn.\ \eqref{sp4wh2}.
Let $U$ continue to denote the standard maximal unipotent
subgroup of $Sp_4$. Denote by $\psi_{U,1}$ the character of $U$
defined by $\psi_{U,1}(u)=\psi(u_{1,2})$, as in \eqref{sp4wh2}.
For $\theta\in\Theta_{Sp_4}^{(3)}$
let $\theta^{U,\psi}$ be the
corresponding coefficient; it
defines a functional $\ell$ on the space of
$\Theta_{Sp_4}^{(3)}$. This functional may be viewed as the convolution of the
intertwining operator with the Whittaker
coefficient on the $GL_2$ part. Note that because of the lack of
uniqueness of the Whittaker model, the functional
$\ell$ is not factorizable.  However, the intertwining operator
is factorizable and because of the uniqueness of the
Whittaker functional at the infinite places, we can separate out
the functional at the infinite places. Let $\ell_\Phi$ denote the
corresponding functional defined on $\otimes_{\nu\in\Phi} \big({\Theta_{Sp_4}^{(3)}}\big)_\nu$.
Then $L_\Phi$ is defined by $L_\Phi(g)=\ell_\Phi(\rho(g)\theta)$ where
$\rho$ is right translation.  With this definition, using \eqref{sp4wh2}, we see that \eqref{factor1} holds.

Thus we have
$L_\Phi(uh)=\psi_{U,1}(u)L_\Phi(h)$ for all $u\in U(\mathbb{A}_\Phi)$, $h\in Sp_4^{(3)}(\mathbb{A}_\Phi)$.
Also, since we assume that $L_\Phi$
is obtained from a factorizable vector $\theta$, there is a compact open group
$K_\theta:=\prod_{\nu\in\Phi}K_\nu$ such that $L_\Phi(hk)=L_\Phi(h)$, where
for almost all places $K_\nu$ is the maximal compact subgroup $Sp_4(O_\nu)$ of
$Sp_4(F_\nu)$, and at the remaining places $K_\nu$ is a
principal congruence subgroup.  Let $T$ denote the maximal torus of
$Sp_4$ which consists of diagonal elements.
We shall now choose a maximal abelian subgroup of $\widetilde{T}(\mathbb{A}_\Phi)$.
Let $\widetilde{T}_\nu$
denote the inverse image of $T_\nu$ in $Sp_4^{(3)}(F_\nu)$. For each
place, we choose a maximal abelian subgroup of $\widetilde{T}_\nu$,
denoted $\widetilde{T}^0_\nu$, as follows. Let $p$ denote a generator of the prime ideal
in the ring of integers of $F_\nu$. Then, from the properties of the local Hilbert
symbol $(~,~)_3$ of $F_\nu$,  we have $(p,-p)_3=1$. Since $-1$ is a cube,
it follows that $(p^{k_1},p^{k_2})_3=1$ for all integers $k_1$ and
$k_2$. Let $\widetilde{T}^0_\nu$ denote the group generated by the
center of $\widetilde{T}_\nu$ and the group of all matrices of the
form $\text{diag}(p^{k_1},p^{k_2},p^{-k_2},p^{-k_1})$. Finally, let
$\widetilde{T}^0= \prod_{\nu\in\Phi}\widetilde{T}^0_\nu$.  We fix
this as the maximal abelian subgroup in the $Sp_4$ case from now on.

Given $t=\text{diag}(t_1,t_2,t_2^{-1},t_1^{-1})$, denote $t'=\text{diag}(t_1,t_2)
\in GL_2$. Then it follows from Prop.~\ref{propconstant1} that for $t\in \widetilde{T}^0$ we have
\begin{equation}\label{factor2}
L_\Phi(t)=|t_1t_2|\,\overline{W}_{{\Theta^{(3)}_{GL_2}},\text{fin}}(t')\, L_\Phi(e).
\end{equation}
Here the function $W_{{\Theta^{(3)}_{GL_2}},\text{fin}}$ is the finite part of the Whittaker function
on a suitable vector in the representation $\otimes_{\nu\in\Phi}(\Theta_{GL_2}^{(3)})_\nu$. Also, the
factor $|t_1t_2|$ is equal to $\delta_P^{1-s_0}(t)$ where $s_0=2/3$ is the point where
$\Theta_{Sp_4}^{(3)}$ is defined as a residue of an Eisenstein series. See Section 3.
\end{proof}

The situation in the case of descent from $SO_7^{(4)}$ to $SO_4^{(4)}$ is similar. 
We recall that $-1$ is a fourth power in $F^\times$.
Applying the same reasoning starting from the expression
\eqref{whso74}  
we obtain
\begin{lemma}\label{newlemma2} The Whittaker coefficient \eqref{whso720} has finite 
part given by
\begin{equation}\label{factor3}
|a|_\Phi^{-1}\int\limits_{{\mathbb{A}}_\Phi^3} L_\Phi(z(r_1,r_2)t(a,1,1)
y_2(x))\,\psi(r_1)\,dr_1\,dr_2\,dx,
\end{equation}
where the function $L_\Phi$ is described below and is computed on a certain maximal abelian subgroup of
$\widetilde{T}({\mathbb A}_\Phi)$ by  \eqref{factor4} below.
\end{lemma}

The function $L_\Phi$ in this case  is
obtained from a functional on $\Theta_{SO_7}^{(4)}$ applied to
a vector in this space; the function $L_\Phi$ satisfies the property
$L_\Phi(uh)=\psi_{U,1}(u)\,L_\Phi(h)$ for all $u\in U(\mathbb{A}_\Phi)$, 
$h\in SO_7^{(4)}(\mathbb{A}_\Phi)$, where
now $U$ is the maximal unipotent subgroup of $SO_7$ and
$\psi_{U,1}(u)=\psi(u_{1,2}+u_{2,3})$, as in \eqref{four3313}
(recall that $\Theta_{SO_7}^{(4)}$ is not generic; a factorization similar
to Lemma~\ref{lem3} holds for similar reasons). Analogously  to
property \eqref{factor2}, one has
\begin{equation}\label{factor4}
L_\Phi(t)=|t_1t_2t_3|W_{{\Theta^{(4)}_{GL_3}},\text{fin}}(t')\,L_\Phi(e)
\end{equation}
where the group $\widetilde{T}^0$ is defined similarly to the prior case, and
for $t=t(t_1,t_2,t_3 )\in \widetilde{T}^0$, we
denote $t'=\text{diag}(t_1,t_2,t_3)\in GL_3$.

\section{Properties of the Descent Integrals}\label{section5}

In this section we let $H$ be one of the groups $Sp_4$, $SO_7$.
Let $n=3$ if $H=Sp_4$, and $n=4$ if $H=SO_7$. We let $H^{(n)}$
denote the $n$-fold cover of $H$.  In the first case we let
$G=SL_2$, and in the second case we let $G=SO_4$. Let $m=6$ in the
first case and $m=4$ in the second. We also recall that we have fixed
a maximal abelian subgroup $\widetilde{A}(\mathbb{A})$ of $\tilde{T}(\mathbb{A})
\subseteq H^{(n)}(\mathbb{A})$.  We shall specify a specific choice, written $\tilde{T}_0$,
later in this section.
In the previous sections we
constructed nonzero representations $\sigma_G^{(m)}$ of the group
$G^{(m)}({\mathbb{A}})$ acting on functions on $G(F)\backslash G^{(m)}({\mathbb{A}})$, 
and computed their constant term and a certain
Whittaker coefficient. We have

\begin{proposition}\label{prop5-1}
The representations $\sigma_G^{(m)}$ are each in $L^2(G(F)\backslash
G^{(m)}({\mathbb{A}}))$.  Moreover the projection of each representation to the residual spectrum is non-zero.
\end{proposition}

\begin{proof}
The proof that the functions obtained by the descent are square-integrable has two parts.  
First, we shall show that for each such function $f(g)$,
the non-degenerate Whittaker coefficients are
rapidly decreasing functions.  Second, we shall examine the degenerate Whittaker coefficients.

We consider  first the $Sp_4$ case. To show that the non-degenerate Whittaker coefficients are rapidly decreasing,
we begin with the expression \eqref{whittaker1} for these coefficients, but with $-\frac14$ replaced
by a general $\beta\in F^\times$.  Let $g=t(1,a)$.  Moving $g$ to the left in the integrand,
applying the action of the Weil representation, and changing variables, we find
that the non-degenerate Fourier coefficient is given by
\begin{multline*}\gamma(a)\,|a|^{-1/2}
\int\limits_{{\mathbb{A}}}\int\limits_{(F\backslash {\mathbb
A})^3}\phi(x)\,\theta\left (\begin{pmatrix} 1&&y&z\\ &1&r&y\\ &&1&\\
&&&1\end{pmatrix}t(1,a)\begin{pmatrix} 1&x&&\\ &1&&\\ &&&1&-x\\
&&&&1\end{pmatrix} \right )\\ \times 
\psi(z+\beta r)\,\,dy\,dz\,dr\,dx
\end{multline*}
The convolution of $\phi$ with $\theta$ obtained from the $x$ integration gives another function in the theta space and so
it is sufficient to show that the integral
$$
\int\limits_{(F\backslash {\mathbb
A})^3}\theta\left (\begin{pmatrix} 1&&y&z\\ &1&r&y\\ &&1&\\
&&&1\end{pmatrix}t(1,a)\right ) \psi(z+\beta r)\,\,dy\,dz\,dr
$$
is rapidly decreasing in $a$ for any function $\theta$ in the space of $\Theta_{Sp_4}^{(3)}$.

By the theorem of Dixmier-Malliavin \cite{D-M} we may suppose that $\theta$ is a convolution 
$$\theta(g)=\int_{\mathbb A} \theta(g x_\alpha(m))\,\phi_1(m)\,dm$$
where $x_\alpha(m)=I_4+m e_{2,3}$ 
and $\phi_1$ is a Schwartz-Bruhat function on $\mathbb A$.  Substituting in, conjugating $x_\alpha(m)$ to the left,
and changing variables, we see that the last integral becomes
$$
\hat{\phi}_1(a^2)\int\limits_{(F\backslash {\mathbb
A})^3}\theta\left (\begin{pmatrix} 1&&y&z\\ &1&r&y\\ &&1&\\
&&&1\end{pmatrix}t(1,a)\right ) \psi(z+\beta r)\,\,dy\,dz\,dr,
$$
where $\hat\phi_1$ is the Fourier transform of $\phi_1$.
Since this Fourier transform is again a Schwartz-Bruhat function while $\theta$ grows at most polynomially
in $|a|$, this expression implies that
the Whittaker coefficients attached to a non-degenerate character are indeed rapidly decreasing in $a$.

Since the non-zero Whittaker coefficients are rapidly decreasing, to see that functions in the descent space
are square integrable, it suffices to examine the constant term.
Here we use the identity of
integral \eqref{sp44} with integral \eqref{sp45}, and
the identity in Prop.~ \ref{propconstant1}, part (ii). It
follows from these results that if $\text{diag}(a,1,1,a^{-1})$ lies
in $\widetilde{A}({\mathbb{A}})$, then the integral
\eqref{sp45} is equal to
\begin{equation}\label{spec1}
|a|^{5/6}\gamma(a)\int\limits_{{\mathbb{A}}}\phi(x)\theta^{U,\psi_U}\left (
w_1\begin{pmatrix} 1&x&&\\ &1&&\\ &&&1&-x\\
&&&&1\end{pmatrix} \right )dx.\notag
\end{equation}
From this we deduce the identity
\begin{equation}\label{spec2}
\int\limits_{F\backslash {\mathbb{A}}}f_{SL_2}^{(6)}\left (\begin{pmatrix} 1&r\\
&1\end{pmatrix}\begin{pmatrix} a&\\ &a^{-1}\end{pmatrix}\right )dr=
|a|^{5/6}\gamma(a)\int\limits_{F\backslash {\mathbb{A}}}f_{SL_2}^{(6)}\begin{pmatrix} 1&r\\
&1\end{pmatrix}dr\notag
\end{equation}
for all diagonal matrices which lie in the corresponding maximal abelian subgroup
of the diagonal elements of $SL_2$.  This means that the constant
term defines an $SL_2^{(6)}({\mathbb{A}})$ mapping from the
representation $\sigma_{SL_2}^{(6)}$ into the representation
$Ind_{\widetilde{B}({\mathbb{A}})}^{SL_2^{(6)}({\mathbb
A})}\delta_B^{-1/12}\delta_B^{1/2}$. Indeed, it follows from the
discussion  at the beginning of Section~\ref{definitions} that the
induced representation is determined uniquely by the values of the
character on a maximal abelian subgroup of the diagonal
elements.

It now follows directly from the Fourier expansion on $SL_2$
that each function in $\sigma_{SL_2}^{(6)}$ is square-integrable. 
Also, since this representation has a nonzero constant term, it 
is not contained in the cuspidal subspace.  In fact, the constant term has negative exponents, 
so may not be obtained from the continuous spectrum.  We conclude that $\sigma_{SL_2}^{(6)}$
has non-zero projection to the residual spectrum.  

We remark that if we knew that functions in 
$\sigma_{SL_2}^{(6)}$ were $\mathcal Z$-finite at infinity, where $\mathcal Z$ is the center of the universal
enveloping algebra, then we would be able to deduce these results purely from the computation of
the exponents of the constant term, by Jacquet's criterion. In that case the non-degenerate Whittaker coefficients
would automatically be rapidly decreasing.

The case when $H=SO_7$ is similar, but slightly more complicated.  For the non-degenerate Whittaker coefficients, 
we must study the growth of
\begin{equation*}
\int\limits_{(F\backslash {\mathbb{A}})^2}f_{SO_4}^{(4)}\left
(\begin{pmatrix} 1&r_1&r_2&-r_1r_2\\ &1&&-r_2\\ &&1&-r_1\\
&&&1\end{pmatrix}
\begin{pmatrix} ab&&&\\ &a&&\\ &&a^{-1}&\\ &&&a^{-1}b^{-1}\end{pmatrix}
\right )\,\psi(r_1-\frac{1}{2}r_2) \, dr_1\, dr_2
\end{equation*}
in $a$ and $b$.  Once again applying the theorem of Dixmier-Malliavin, we may study the growth of
\begin{multline*}
\int\limits_{{\mathbb A}^2}
\int\limits_{(F\backslash {\mathbb{A}})^2}f_{SO_4}^{(4)}\left
(\begin{pmatrix} 1&r_1&r_2&-r_1r_2\\ &1&&-r_2\\ &&1&-r_1\\
&&&1\end{pmatrix}
\begin{pmatrix} ab&&&\\ &a&&\\ &&a^{-1}&\\ &&&a^{-1}b^{-1}\end{pmatrix}
\begin{pmatrix} 1&m_1&m_2&*&\\&1&&-m_2\\&&1&-m_1\\&&&1\end{pmatrix}
\right )\\
\times \phi_1(m_1)\,\phi_2(m_2)\,\psi(r_1-\frac{1}{2}r_2) \, dr_1\, dr_2\,dm_1\,dm_2,
\end{multline*}
where $\phi_1$, $\phi_2$ are Schwartz-Bruhat functions.
Moving the matrix in $m_1,m_2$ leftward and changing variables, we arrive at 
\begin{multline*} \hat{\phi}_1(b)\,\hat{\phi}_2(a^2b)\,
\int\limits_{(F\backslash {\mathbb{A}})^2}f_{SO_4}^{(4)}\left
(\begin{pmatrix} 1&r_1&r_2&-r_1r_2\\ &1&&-r_2\\ &&1&-r_1\\
&&&1\end{pmatrix}
\begin{pmatrix} ab&&&\\ &a&&\\ &&a^{-1}&\\ &&&a^{-1}b^{-1}\end{pmatrix}
\right )\\ \times\psi(r_1-\frac{1}{2}r_2) \, dr_1\, dr_2\,dm_1\,dm_2,
\end{multline*}
Since the product of the Fourier transforms is rapidly decreasing in $|a|$, $|b|$, while the remaining term is a compact integral of an
automorphic form and hence at most slowly increasing in $|a|$, $|b|$, we conclude that the non-degenerate Whittaker
coefficients are indeed rapidly decreasing.

As for the constant term, 
starting with the fact that
integral \eqref{so72} is equal to integral \eqref{so74}, we then use the
two last parts of Prop.~\ref{propconstant2}. We deduce the
equality
\begin{equation}\label{spec3}
\int\limits_{(F\backslash {\mathbb{A}})^2}f_{SO_4}^{(4)}\left (l(r_1,r_2)
\begin{pmatrix} ab&&&\\ &b&&\\ &&b^{-1}&\\ &&&a^{-1}b^{-1}\end{pmatrix}
\right ) dr_1\,dr_2=|ab|^{3/4}\,\int\limits_{(F\backslash {\mathbb
A})^2}f_{SO_4}^{(4)}(l(r_1,r_2))\,dr_1\,dr_2\notag
\end{equation}
where
$$l(r_1,r_2)=\begin{pmatrix}
1&r_1&r_2&-r_1r_2\\ &1&&-r_2\\ &&1&-r_1\\ &&&1\end{pmatrix}$$ and
$\text{diag}(ab,b,b^{-1},a^{-1}b^{-1})$, embedded via \eqref{so73}, lies in $\widetilde{A}(\mathbb{A})$.
Hence, the constant
term maps $\sigma_{SO_4}^{(4)}$ into the principal series $Ind_{\widetilde{B}({\mathbb
A})}^{SO_4^{(4)}({\mathbb{A}})}\delta_B^{-1/8}\delta_B^{1/2}$. 

There are also two intermediate terms, corresponding to the characters of that are supported on exactly
one of $r_1$, $r_2$.  Arguing as above, we see that these coefficients are rapidly decreasing in one variable $a,b$
and slowly increasing in the other.  Then the Fourier expansion on $SO_4$ 
shows
that each function in $\sigma_{SO_4}^{(4)}$ is square-integrable.   Once again,
since this representation has a constant term, it 
is not contained the subspace of cuspidal $L^2$ functions, and since the constant term has negative exponents, 
it may not be obtained from the continuous spectrum.  We conclude that $\sigma_{SO_4}^{(4)}$
has non-zero projection to the residual spectrum.  
\end{proof}

At this point, we do not know that the descents are themselves automorphic.  This is essentially
a question about archimedean Whittaker functions, and would follow if we could show that the descent
at the archimedean places yielded $\mathcal Z$-finite functions.  Though  we only know that the descent
projects non-trivially to the space of $\Theta_G^{(m)}$, there are other examples where a descent construction is
both automorphic and irreducible (Jiang-Soudry \cite{J-S}).  Hence we make the following conjecture.

\begin{conjecture}\label{conjecture1}
For each of the descents treated here,  $\sigma_G^{(m)}=\Theta_G^{(m)}$.
\end{conjecture}

We shall see evidence for this conjecture below, in that there are functions in the space of $\sigma_G^{(m)}$
whose properties exactly match the conjectures of Patterson and Chinta-Friedberg-Hoffstein.

\section{ Some Local Computations}\label{section6}

Let us introduce the following notation for
the normalized $n$-th order Gauss sums.  Suppose that $\nu$ is a finite place with local
uniformizer $p$, $|p^{-1}|_\nu=q,$ $\psi_\nu$ is an additive character of conductor $O_\nu$,
$|b|_\nu\leq1$,  and $j\in\mathbb{Z}$, $(j,n)=1$.  Then we write
$$\int\limits_{|\epsilon|_\nu=1}(\epsilon,p)_n^j\,\psi_\nu(bp^{-1}\epsilon)\,d\epsilon=q^{-1/2}G_j^{(n)}(b,p).$$
Here $G_j$ is a normalized $n$-th order Gauss sum modulo $p$.
We extend the notation to composite moduli as in \cite{C-F-H}, pg.\ 151; thus if $p,q$ are local uniformizers
at places $\nu_1,\nu_2$ resp., then
\begin{equation}\label{gauss9}
G_j^{(n)}(b,pq)=(p,q)^j_{\nu_2}
(q,p)^j_{\nu_1}G_j^{(n)}(b,p)G_j^{(n)}(b,q),
\end{equation}
where we have indicated the fields for the two local residue symbols.

Fix $\nu$ to be a finite place such that all data is unramified at $\nu$.
In Lemmas~\ref{newlemma1}, \ref{newlemma2} we have obtained expressions for the finite parts of the
Whittaker coefficients of the descents.
In this section we will compute the integrals
\eqref{factor1} and \eqref{factor3} at $a=(a_\nu)_\Phi\in {\mathbb A}^\times_\Phi$
defined as follows. For all $\nu'\ne \nu$ we assume that $a_{\nu'}$
is a unit. At the place $\nu$ we let $a_\nu$ be a positive power of a local uniformizer
$p_\nu$.  Let $\theta=\otimes\theta_{\nu'}$ be a factorizable
vector whose Whittaker coefficients are not identically zero.  (The Whittaker functionals
can not be identically zero on all factorizable vectors by Lemmas~\ref{nonzerosp4}, \ref{nonzeroso7}.) 
Notice that in both cases, we have $t(a,1)$ and $t(a,1,1)$ are
in the corresponding groups $\widetilde{T}^0$, and that outside of
$\nu$ the element $p_\nu$ is a unit.

Consider first the $Sp_4$ case.
We can ignore the integration over $x$ in the integral
\eqref{factor1} in the following sense. In the previous Section we defined the
compact group $K_\theta=\prod_{\nu'\in\Phi}K_{\nu'}$ such that
$L_\Phi(hk)=L_\Phi(h)$. Choose a Schwartz function $\phi$ such that
$\phi_{\nu'}$ is one if $|x_{\nu'}|_{\nu'}$ is so small that
$m(x_{\nu'}+a_{\nu'}/2)\in K_{\nu'}$, and zero otherwise. Then the integral
\eqref{factor1} is a positive multiple of
\begin{equation}\label{finwh3}
|a|_\Phi^{-1/2}\gamma(a)\int\limits_{{\mathbb
A}_\Phi}L_\Phi(z(r)t(a,1)w_2w_1)\,\psi(r)\,dr.
\end{equation}
(The map $\kappa$ defined
at the end of Section~\ref{basic-notations} is identically
1 on unipotent matrices, so no roots of unity are introduced here.)

In the $SO_7$ case, we argue as follows. Given a place $\nu'$ which is
unramified, let $y_3(m)=I_7+me_{1,6}'$ such that $y_3(m)\in K_{\nu
'}$. Consider the integral
\begin{equation}\label{finwh1}
Z(x_{\nu'})=\int\limits_{{\mathbb{A}}_{\Phi}^2}
L_\Phi(z(r_1,r_2)t(a,1,1) y_2(x_{\nu'}))\,\psi(r_1)\,dr_1\,dr_2.\notag
\end{equation}
Using the right invariant property of $L_\Phi$ at the place $\nu'$, we have
\begin{equation}\label{finwh2}
Z(x_{\nu'})=\int\limits_{{\mathbb{A}}_{\Phi}^2}
L_\Phi(z(r_1,r_2)t(a,1,1) y_2(x_{\nu'})y_3(m))\,\psi(r_1)\,dr_1\,dr_2.\notag
\end{equation}
Conjugating $y_3(m)$ to the left and using the left-invariance
properties of $L_\Phi$ under the group $V$, we obtain
$Z(x_{\nu'})=\psi(ma_{\nu'}x_{\nu'})Z(x_{\nu'})$.  From the
definition of $a$, we know that $a_{\nu'}$ is a unit.
Thus $Z(x_{\nu'})$ is zero
unless $|x_{\nu'}|\leq1$. In the bad places we may argue similarly to see
that the integral \eqref{factor2} is zero if $|x_{\nu'}|$ is large.
Since the function $Z(x_{\nu'})$ is locally constant, this
means that the integration over $x_{\nu'}\in F_{\nu'}$ can be
replaced by a finite sum of right translations, leading to an adjustment of the function
$L_\Phi$.

We now consider the $Sp_4$ case and the integral
\eqref{finwh3}. We claim that by a suitable change involving $L_\Phi$ we
can replace the integration over ${\mathbb{A}}_\Phi$ with an integration
over $F_\nu$. This follows from an argument similar to the one above.
More precisely, for an unramified  place $\nu'\ne\nu$,
define the function $Z_1(x_{\nu'})=L_\Phi(z(x_{\nu'})t(a,1))$. Let
$z_1(m)=I_4+me_{1,3}'$ with $m\in F_{\nu'}$, $|m|_{\nu'}\leq1$.  Then
$z_1(m)\in K_{\nu'}$. Arguing as above we obtain
$Z_1(x_{\nu'})=\psi(ma_{\nu'}x_{\nu'})\,Z_1(x_{\nu'})$. Again, from the
definition of $a$, we know that $a_{\nu'}$ is a unit. Hence
if $Z_1(x_{\nu'})\neq0$ then $z(x_{\nu'})\in K_{\nu'}$. Similarly, in a ramified
place $\nu'$, the integral is zero if $|x_{\nu'}|_{\nu'}$ is sufficiently large.
Recall that the function $L_\Phi$ is obtained by
applying the functional $\ell_\Phi$
to a suitable vector $\theta\in \Theta_{Sp_4}^{(3)}$
(see the discussion following \eqref{factor1}).  Replacing $\theta$ by a suitable finite sum of translations
which involve the bad places only and then applying this functional, we thus obtain a new function  which we
denote by $\widetilde{L}_\Phi$, such that the integral \eqref{finwh3} is
equal to
\begin{equation}\label{finwh4}
|a|_\nu^{-1/2}\gamma(a)\int\limits_{F_\nu}\widetilde{L}_\Phi(z(r_\nu)t(a,1))\,\psi_{\nu}(r_\nu)\,dr_\nu.
\end{equation}
Here $\psi_\nu$ is the local constituent of $\psi$.

To simplify notation we drop the subscript $\nu$.
To compute the integral \eqref{finwh4}, we first conjugate
$t(a,1)$ to the right. Then we write this integral as a sum of two
integrals, integrating separately over the regions $|r|\le 1$ and $|r|>1$. Since $\nu$ is an
unramified place, we may dispose of the $r$ variable in the first
integral. Using the property \eqref{factor2}, the contribution of $|r|\leq1$ to
integral \eqref{finwh4} is
$$|a|^{1/2}\gamma(a)\,\overline{W}_{{\Theta^{(3)}_{GL_2}},\text{fin}}\begin{pmatrix} a& \\&1\end{pmatrix}\,
\widetilde{L}_\Phi(e).$$

In the summand that
corresponds to the integration over $|r|>1$, we perform the Iwasawa decomposition of
\begin{equation}\label{iwas}
z(r)=\begin{pmatrix} 1&&&\\ &1&&\\ &r&1&\\ &&&1\end{pmatrix}=\begin{pmatrix}
1&&&\\ &1&r^{-1}&\\ &&1&\\ &&&1\end{pmatrix}\begin{pmatrix} 1&&&\\ &r^{-1}&&\\ &&r&\\&&&1\end{pmatrix}k
\end{equation}
where $k$ is in the maximal compact subgroup. It is easy to check that $\kappa(k)=1$.
This implies that the second integral is equal to
$$|a|^{-1/2}\gamma(a)\int\limits_{|r|>1}\widetilde{L}_\Phi\left(
\begin{pmatrix} a&&&\\ &1&&\\ &&1&\\ &&&a^{-1}\end{pmatrix}
\begin{pmatrix} 1&&&\\ &r^{-1}&&\\ &&r&\\ &&&1\end{pmatrix}\right)
\psi(r)\,dr.$$
Write $r=p^{-m}\epsilon$ where $m\geq1$ and $|\epsilon|=1$. Then the factorization
\begin{equation}\label{non-trivial-cocycle}\begin{pmatrix} a&&&\\ &1&&\\ &&1&\\ &&&a^{-1}\end{pmatrix}
\begin{pmatrix} 1&&&\\ &p^m\epsilon^{-1}&&\\ &&p^{-m}\epsilon &\\ &&&1\end{pmatrix}=
\begin{pmatrix} a&&&\\ &p^m&&\\ &&p^{-m}&\\ &&&a^{-1}\end{pmatrix}
\begin{pmatrix} 1&&&\\ &\epsilon^{-1}&&\\ &&\epsilon &\\ &&&1\end{pmatrix}
\end{equation}
contributes
a factor of $(\epsilon,p^m)_3$ due to the 2-cocycle. (Since $(p,p)_3=1$ we do not record an additional factor
of $(a,p^m)_3$.) Thus we find that the second integral is equal to
\begin{equation}\label{double-int}
|a|^{-1/2}\gamma(a)\sum_{m=1}^\infty q^m\widetilde{L}_\Phi\begin{pmatrix} a&&&\\ &p^m&&\\ &&p^{-m}&\\ &&&a^{-1}\end{pmatrix}
\int\limits_{|\epsilon|=1}(\epsilon,p^m)_3\,\psi(p^{-m}\epsilon)\,d\epsilon.
\end{equation}
If $m>1$, the inner integral in \eqref{double-int} is zero, and if $m=1$ it gives
$$\int\limits_{|\epsilon|=1}(\epsilon,p)_3\,\psi(p^{-1}\epsilon)\,d\epsilon=q^{-1/2}G_1^{(3)}(1,p).$$
Let
$\iota(p)$ be the finite idele which is $1$ outside of $\nu$ and $p$ at $\nu$.
Then we have shown that the second integral gives a contribution of
$$ |a|^{1/2}\gamma(a)\,q^{-1/2}\,G_1^{(3)}(1,p)\,\overline{W}_{{\Theta^{(3)}_{GL_2}},\text{fin}}\begin{pmatrix} a& \\&\iota(p)\end{pmatrix}\,\widetilde{L}_\Phi(e).$$

Given a Whittaker function $W_{\pi_M}$ attached to a representation 
$\pi_M$ on a group $M$ we define the normalized Whittaker coefficient
$\tau_{\pi_M}(t)=\delta_B^{-1/2}(t)W_{\pi_M}(t)$ where
$B$ is the Borel subgroup of $M$. 
This normalization is consistent with the theta coefficients described (in a non-adelic language) in \cite{C-F-H}.
(Here we suppress the choice of vector used to define $W_{\pi_M}$ from the notation.)
 
We summarize the results obtained here in the following Proposition.
Recall that we have fixed a local uniformizer $p_\nu$ at $\nu$.
\begin{proposition}\label{prop61}  Suppose that $\nu$ is an unramified place.
There is a vector in $\sigma^{(6)}_{SL_2}$ with $\tau_{{\sigma^{(6)}_{SL_2}},\text{fin}}(e)\ne0$
and a vector in ${\Theta^{(3)}_{GL_2}}$
such that for all finite ideles $a$ which are units outside of $\nu$ and a positive power of the local uniformizer
$p_\nu$ at $\nu$, the following equality of Whittaker coefficients holds:
\begin{equation}\label{finwh5}
\tau_{{\sigma^{(6)}_{SL_2}},\text{fin}}\begin{pmatrix} a&
\\&a^{-1}\end{pmatrix}=
\gamma(a)\left (\overline{\tau}_{{\Theta^{(3)}_{GL_2}},\text{fin}}\begin{pmatrix} a&
\\&1\end{pmatrix}+
G_1^{(3)}(1,p_\nu)\overline{\tau}_{{\Theta^{(3)}_{GL_2}},\text{fin}}\begin{pmatrix}
a& \\&\iota(p_\nu)\end{pmatrix}\right )\,\tau_{{\sigma^{(6)}_{SL_2}},\text{fin}}(e).
\end{equation}
\end{proposition}
\noindent 
Here the coefficients $\tau_{{\sigma^{(6)}_{SL_2}},\text{fin}}$ and $\tau_{{\Theta^{(3)}_{GL_2}},\text{fin}}$
depend on choices of vector but we have suppressed this from the notation.  
The factor $\tau_{{\sigma^{(6)}_{SL_2}},\text{fin}}(e)$ on the right hand side 
of \eqref{finwh5} comes from $\widetilde{L}_\Phi(e)$.  Note also that if we choose
the vector in $\sigma^{(6)}_{SL_2}$ so that $\tau_{{\sigma^{(6)}_{SL_2}},\text{fin}}(e)=1$ then it follows from \eqref{finwh5}
that the corresponding vector in ${\Theta^{(3)}_{GL_2}}$ satisfies $\tau_{{\Theta^{(3)}_{GL_2}},\text{fin}}(e)=1.$

Next we analyze the $SO_7$ integral \eqref{factor3}.
After omitting the integration over $x$, we consider the integral
$$|a|_\Phi^{-1}\int\limits_{{\mathbb{A}}_\Phi^2} L_\Phi(z(r_1,r_2)t(a,1,1)
)\,\psi(r_1)\,dr_1\,dr_2.$$ Arguing as in the $Sp_4$ case, immediately after Eqn.~\eqref{finwh3}, we may
reduce to the computation of the local integral
$$|a|_\nu^{-1}\int\limits_{F_\nu^2} \widetilde{L}_\Phi(z(r_1,r_2)t(a,1,1)
)\,\psi_\nu(r_1)\,dr_1\,dr_2,$$
where $\widetilde{L}_\Phi$ is obtained from a suitable sum of right translates,
similarly to the $Sp_4$ case above.
We omit the place $\nu$ from the notation. Consider first the integration over $|r_1|\le 1$.
From the right-invariance property of $L_\Phi$ we obtain that
$\widetilde{L}_\Phi(z(r_1,r_2)t(a,1,1))=\widetilde{L}_\Phi(z(0,r_2)t(a,1,1))$.
In this case we claim that we may also limit the domain of integration to $|r_2|\le 1$.
Indeed, if $|r_2|>1$, then we use the right-invariance of
$L_\Phi$ by the maximal compact subgroup (at $\nu$)
to deduce that $\widetilde{L}_\Phi(h)=\widetilde{L}_\Phi(hy(m))$ for
$y(m)=I_7+me_{2,4}-\frac{m^2}{2}e_{2,6}$ with $|m|\le 1$. Conjugating this matrix from right to left, we deduce that
$\widetilde{L}_\Phi(z(0,r_2)t(a,1,1))=0$ in this case.  Hence using \eqref{factor4} we obtain the contribution of
$$W_{{\Theta^{(4)}_{GL_3}},\text{fin}}\begin{pmatrix} a&& \\&1&\\ &&1\end{pmatrix}\,\widetilde{L}_\Phi(e)=|a|\,
\tau_{{\Theta^{(4)}_{GL_3}},\text{fin}}\begin{pmatrix} a&& \\&1&\\ &&1\end{pmatrix}\,\widetilde{L}_\Phi(e)$$
where $\tau$ is the normalized Whittaker coefficient of the theta function.

Next consider the domain $|r_1|>1$. Factoring $z(r_1,r_2)=z(0,r_2)z(r_1,0)$ and performing an Iwasawa decomposition
on $z(r_1,0)$ similarly to \eqref{iwas}, we obtain
the contribution of
$$|a|^{-1}\int\limits_F\int\limits_{|r_1|>1}\widetilde{L}_\Phi(t(a,1,1)t_1(r_1^{-1})z(0,r_2))\,
\psi(r_1)\,\psi(r_2^2r_1)\,|r_1|\,dr_1\,dr_2.$$
Here $t_1(r_1^{-1})=\text{diag}(1,r_1^{-1},r_1^{-1},1,r_1,r_1,1)$
and the factor of $\psi(r_2^2r_1)$ is obtained from the conjugation of the unipotent matrix obtained
from the Iwasawa decomposition to the left and a change of variables. 

As above, using the matrix $y(m)$ we deduce that the integral vanishes unless $|r_2|\le 1$. Thus we obtain
$$|a|^{-1}\int\limits_{|r_1|>1}\widetilde{L}_\Phi(t(a,1,1)t_1(r_1^{-1}))\,\psi(r_1)\,|r_1|\int\limits_{|r_2|\le 1}
\psi(r_2^2r_1)\,dr_2\,dr_1.$$
Write $r_1=p^{-m}\epsilon$ with $m\geq1$ and $|\epsilon|=1$.
The factorization $t_1(p^{-m}\epsilon)=t_1(p^{-m}\,)t_1(\epsilon)$
contributes the factor $(\epsilon,p)^{2m}_8=(\epsilon,p)_4^m$ from the 2-cocycle. We thus obtain
\begin{equation}\label{finwh6}
|a|^{-1}\sum_{m=1}^\infty q^{2m}\widetilde{L}_\Phi(t(a,1,1)t_1(p^m))\int\limits_{|\epsilon|=1}\int\limits_{
|r_2|\le 1}(\epsilon,p)_4^m\,\psi(p^{-m}\epsilon)\,\psi(r_2^2p^{-m}\epsilon)\,dr_2\,d\epsilon.
\end{equation}
Writing $r_2=p^n\eta$ with $|\eta|=1$, the inner integration is equal to
\begin{equation}\label{finwh7}
\sum_{n=0}^\infty q^{-n}\int\limits_{|\epsilon|=1}\int\limits_{|\eta|=1}
(\epsilon,p)_4^m\,\psi(p^{-m}\epsilon)\,\psi(p^{2n-m}\eta^2\epsilon)\,d\eta \,d\epsilon.
\end{equation}
If $n\ge 1$, we notice that $\psi(p^{-m}\epsilon)\psi(p^{2n-m}\eta^2\epsilon)=
\psi(p^{-m}\epsilon(1+p^{2n}\eta^2))$. Since, for $n\ge 1$, $1+p^{2n}\eta^2$ is a unit which is
a fourth power, we can make a change of variables in $\epsilon$ to remove this factor. Therefore, summing
over $n\ge 1$ we obtain the contribution
$$\sum_{n=1}^\infty q^{-n}(1-q^{-1})\int\limits_{|\epsilon|=1}(\epsilon,p)_4^m\,\psi(p^{-m}\epsilon)\,d\epsilon
=
\begin{cases}q^{-3/2}G_1^{(4)}(1,p)&\text{if $m=1$}\\0&\text{if $m>1.$}
\end{cases}$$
Here the inner integral vanishes if $m>1$ due to the oscillation in $\psi$.
Thus the contribution to the integral \eqref{finwh6} from the summands with $n\geq1$ is
$$|a|\,q^{-1/2}G_1^{(4)}(1,p)\,
\tau_{{\Theta^{(4)}_{GL_3}},\text{fin}}\begin{pmatrix} a&& \\&\iota(p)&\\ &&\iota(p)\end{pmatrix}
\,\widetilde{L}_\Phi(e).$$
Finally, we consider the contribution to the sum \eqref{finwh7} from the term $n=0$. It is equal to
$$\int\limits_{|\epsilon|=1}\int\limits_{|\eta|=1}
(\epsilon,p)_4^m\,\psi(p^{-m}\epsilon)\,\psi(p^{-m}\eta^2\epsilon)\,d\eta \,d\epsilon.$$
We use the identity
$$\int\limits_{|\eta|=1}\psi(p^{-m}\epsilon\eta^2)\,d\eta=\int_{|\eta|=1}\psi(p^{-m}\epsilon\eta)\,(1+(\eta,p)_4^2)\,d\eta.$$
Plugging this into the above integral and changing variables $\eta\mapsto \eta\epsilon^{-1}$, we obtain
$$\int\limits_{|\epsilon|=1}(\epsilon,p)_4^m\,\psi(p^{-m}\epsilon)\,d\epsilon\int\limits_{|\eta|=1}\psi(p^{-m}\eta)\,d\eta+
\int\limits_{|\epsilon|=1}(\epsilon,p)_4^{3m}\,\psi(p^{-m}\epsilon)\,d\epsilon\int\limits_{|\eta|=1}(\eta,p)_4^2\,\psi(p^{-m}\eta)
\,d\eta.$$
If $m>1$ this expression is zero, and for $m=1$ we obtain
$-q^{-3/2}\,G_1^{(4)}(1,p)+q^{-1}\,G_3^{(4)}(1,p).$ Therefore the contribution to \eqref{finwh6} is
$$\left(G_3^{(4)}(1,p)\,\tau_{{\Theta^{(4)}_{GL_3}},\text{fin}}\begin{pmatrix} a&& \\&\iota(p)&\\ &&\iota(p)\end{pmatrix}-
q^{-1/2}\,G_1^{(4)}(1,p)\,\tau_{{\Theta^{(4)}_{GL_3}},\text{fin}}\begin{pmatrix} a&& \\&\iota(p)&\\ &&\iota(p)\end{pmatrix}
\right)\widetilde{L}_\Phi(e).$$

Combining all terms we obtain the following Proposition.
\begin{proposition}\label{prop62}
Suppose that $\nu$ is an unramified place.
There is a vector in $\sigma^{(4)}_{SO_4}$ with $\tau_{{\sigma^{(4)}_{SO_4}},\text{fin}}(e)\ne0$
and a vector in ${\Theta^{(4)}_{GL_3}}$
such that for all finite ideles $a$ which are units outside of $\nu$ and a positive power of the local uniformizer
$p_\nu$ at $\nu$, the following equality of Whittaker coefficients holds:\begin{multline*}\label{finwh8}
\tau_{{\sigma^{(4)}_{SO_4}},\text{fin}}\begin{pmatrix} a&&& \\&1&&\\ &&1&\\ &&&a^{-1}\end{pmatrix}= \\ \left(
\tau_{{\Theta^{(4)}_{GL_3}},\text{fin}}\begin{pmatrix} a&& \\&1&\\ &&1\end{pmatrix}+G_3^{(4)}(1,p_\nu)\,
\tau_{{\Theta^{(4)}_{GL_3}},\text{fin}}\begin{pmatrix} a&& \\&\iota(p_\nu)&\\ &&\iota(p_\nu)\end{pmatrix} \right)
\tau_{{\sigma^{(4)}_{SO_4}},\text{fin}}(e).
\end{multline*}
\end{proposition}
\noindent Here, as in Prop.~\ref{prop61}, the coefficients $\tau_{{\sigma^{(4)}_{SO_4}},\text{fin}}$ 
and $\tau_{{\Theta^{(4)}_{GL_3}},\text{fin}}$
depend on choices of vector which we have suppressed from the notation, and if we choose
the vector in $\sigma^{(4)}_{SO_4}$ so that $\tau_{{\sigma^{(4)}_{SO_4}},\text{fin}}(e)=1$ then 
the corresponding vector in ${\Theta^{(4)}_{GL_3}}$ satisfies $\tau_{{\Theta^{(4)}_{GL_3}},\text{fin}}(e)=1.$

Props.~\ref{prop61} and \ref{prop62} easily extend to $a$ which are units outside of $\nu$
and powers of $p_\nu$ times units at $\nu$ by $K_\nu$-invariance, taking into account the relevant cocycles.
We shall extend them further below.

\section{Descent Integrals and the CFH Conjecture}

The goal of this section is to prove a result related to Conjecture 1 of \cite{C-F-H}, which describes the
square coefficients of the theta function on the 6-fold cover of $GL_2$.  We
shall show that one of the descent integrals considered above gives a function in $L^2(SL_2^{(6)}(F)\backslash SL_2^{(6)}(\mathbb A))$
whose Fourier coefficients are exactly consistent with the conjecture,
at least at coefficients which are products
of sufficiently nice primes.  This strongly suggests that the Conjecture there is true, perhaps up to some factors
at bad places, and moreover gives support for Conjecture~\ref{conjecture1} above.

We begin with a factorizable vector $\theta$ such that
$\tau_{\sigma_{SL_2}^{(6)},fin}(e)\neq0$.
After possibly multiplying by a constant we suppose that
$\tau_{\sigma_{SL_2}^{(6)},fin}(e)=1$.  
We then seek to study the integral
\begin{equation}\label{cfh1}
W_{\sigma_{SL_2}^{(6)},fin}\begin{pmatrix} a&\\
&a^{-1}\end{pmatrix}= |a|_\Phi^{-1/2}\gamma(a)\int\limits_{{\mathbb
A}_\Phi}L_\Phi(t(a,1)z(r))\,\psi(r)\,dr
\end{equation}

for certain values of $a\in {\mathbb A}_\Phi^\times$. Indeed, starting with the integral \eqref{factor1},
we have chosen the Schwartz function $\phi$ as was specified near the beginning of
Section~\ref{section6}, and also absorbed the Weyl element $w_2w_1$ into the choice of vector.
As explained in Section~\ref{section5},
the function $L_\Phi$ is defined via a functional applied $\theta$.
Arguing as in Section~\ref{section6}, we choose data such that the integral \eqref{cfh1} is not zero for $a=1$.

Let  $S$ be a finite set of places including all bad finite places for the representation
$\Theta_{Sp_4}^{(3)}$, all finite places for which $\psi$ is not of full conductor, and such that
outside of $S$ the vector $\theta$ is the unramified vector.
We note that we do not have information about the set $S$, which one would like to make as small as possible.
It follows from the proof of
Prop.~\ref{prop61} that
\begin{equation}\label{cfh2}
W_{\sigma_{SL_2}^{(6)},fin}(e)=\int\limits_{F_S}L_\Phi(z(r))\,\psi_S(r)\,dr.
\end{equation}
Here $F_S$ is the product of $F_\nu$ over all $\nu\in S$, and $\psi_S=\prod_{\nu\in S}\psi_\nu$.

As explained above, there is a compact open subgroup $K_\theta=\Pi_{\nu\in \Phi}K_{\nu}$ which
fixes the vector $\theta$, with $K_\nu=Sp_4(O_\nu)$ for $\nu\in\Phi$, $\nu\not\in S$.  Next we specify a set of elements $R\subset {\mathbb A}_\Phi^\times$ defined as follows. Let $\nu$ denote
a place outside of $S$, and let $p\in F$ (or $p_\nu$ if it is necessary to show the place $\nu$ explicitly) denote a generator of the maximal ideal in
$O_\nu$ such that the diagonal matrices $t(p^k,p^l)\in K_{\nu'}$ for all $k,l\geq0, \nu'\in \Phi$, $\nu'\neq\nu$.  This will be true
provided  $|p|_{\nu'}=1$ for all $\nu'\ne\nu$ and $p$ is sufficiently close to 1 at the places $\nu'\in S$.
By Dirichlet's theorem, there are infinitely many such $p$. Given such a $p$,
let $a_\nu=(p,p,\ldots)\in {\mathbb A}_\Phi^\times$.  Also let $|a_\nu|=\# O_\nu/pO_\nu$.
We then define $R$ to be the multiplicative set generated by the elements $a_\nu$ as $\nu$ varies
over places outside of $S$, and extend $|\cdot |$ multiplicatively to $R$.
Notice that if $a,b\in R$ then also $a^kb^l\in R$ for all non-negative integers $k$ and $l$. It
is convenient to denote such an element $a_\nu$ by $\underline{p}$ when the place $\nu$ is clear.
Recall also that $\iota(p)$ denotes
the element in ${{\mathbb A}}_\Phi^\times$ which is $p$ at the $\nu$ component and
1 elsewhere. With these choices, we have $L_\Phi(t(\underline{p}^k,\underline{p}^l))=
L_\Phi(t(\iota(p^k),\iota(p^l)))$ for all $k,l\geq0$.  We emphasize that
the set $R$ depends on the stabilizer of the vector $\theta$ that is chosen above.

The main result in this section is the following.

\begin{theorem}\label{maincfh}
Let $\tau_{\sigma_{SL_2}^{(6)},fin}$ be the finite part of the
normalized Whittaker coefficient attached to a vector $\theta\in \sigma_{SL_2}^{(6)}$ as specified above.
\begin{enumerate} 
\item Let $m_1,m_2\in R$.  Then 
\begin{equation*} 
\tau_{\sigma_{SL_2}^{(6)},fin}
\begin{pmatrix} \underline{m_1}\,\underline{m_2^3}&\\ &\underline{m_1}^{-1}\,\underline{m_2}^{-3}\end{pmatrix}=
|m_2|^{1/2} \tau_{\sigma_{SL_2}^{(6)},fin}\begin{pmatrix} \underline{m_1}&\\ &\underline{m_1}^{-1}\end{pmatrix}.
\end{equation*}
\item
Let $m_1,m_2\in R$ be square-free and coprime.   Then
\begin{equation}\label{cfh-conjecture}
\tau_{\sigma_{SL_2}^{(6)},fin}
\begin{pmatrix} \underline{m_1}\,\underline{m_2^2}&\\ &\underline{m_1}^{-1}\,\underline{m_2}^{-2}\end{pmatrix}=
\gamma(\underline{m_1})\,
G_1^{(3)}(1,m_1)\,G_1^{(3)}(1,m_2)^2\left(\frac{m_2}{m_1}\right)_3^2 \sum_{\substack{d_1,d_2\in R\\m_1=d_1d_2}}\left(\frac{d_2}{d_1}\right)_3.
\end{equation}
\end{enumerate}
\end{theorem}

The first part of the Theorem shows that the descent function 
satisfies the same periodicity property as that satisfied by the theta function on the 6-fold cover of $GL_2$, 
as demonstrated by Kazhdan and Patterson
\cite{K-P}, Section 1.4; see also Hoffstein \cite{Ho}, Prop.\ 5.1.  As for the second part, 
the right hand side of \eqref{cfh-conjecture} is essentially given in Conjecture 1 of \cite{C-F-H}.  However, note that the left hand side
depends on a choice of a specific vector and is valid for products of primes that are sufficiently close to $1$ at the places in $S$.
The quantities also depend on the choice
of the additive character $\psi$.

\begin{proof}
The Theorem is a direct consequence of identity \eqref{finwh5} in Prop.~\ref{prop61}. To explain this succinctly we
carry out the details for $m_1,m_2$ products of small
numbers of primes; the general case follows by exactly the same arguments.
Let $\underline{p}$ and
$\underline{q}$ be two elements in $R$ corresponding to the distinct places $\nu_1$ and $\nu_2$, resp.,
and write $\iota(p$), $\iota(q)$ for the finite ideles which are $1$ outside of places $\nu_1$, $\nu_2$ resp.\
where they are $p$, $q$  resp.  Recall that $\tau_{\sigma_{SL_2}^{(6)},fin}(e)=1$.
Arguing
as in the proof of Prop.~\ref{prop61} we obtain the identities

\begin{equation}\label{cfh3}
\tau_{\sigma_{SL_2}^{(6)},fin}\begin{pmatrix} \underline{p}^i&\\
&\underline{p}^{-i}\end{pmatrix}= \gamma\big(\underline{p}^i\big)\left
(\overline{\tau}_{\Theta_{GL_2}^{(3)},fin}\begin{pmatrix}
\underline{p}^i&\\ &1\end{pmatrix}+G_1^{(3)}(1,p)
\overline{\tau}_{\Theta_{GL_2}^{(3)},fin}\begin{pmatrix}
\underline{p}^i&\\ &\iota(p)\end{pmatrix}\right )
\end{equation}
and

\begin{align}\label{cfh4}
\tau_{\sigma_{SL_2}^{(6)},fin}\begin{pmatrix}
\underline{p}^i\underline{q}^j&\\
&\underline{p}^{-i}\underline{q}^{-j}\end{pmatrix}&=
\gamma\big(\underline{p}^i\underline{q}^j\big)\Bigg(\overline{\tau}_{\Theta_{GL_2}^{(3)},fin}\begin{pmatrix}
\underline{p}^i\underline{q}^j&\\ &1\end{pmatrix}\\
\notag &\qquad +G_1^{(3)}(1,p)
\overline{\tau}_{\Theta_{GL_2}^{(3)},fin}\left (\begin{pmatrix}
\underline{p}^i\underline{q}^j&\\ &1\end{pmatrix}
\begin{pmatrix} 1&\\ &\iota(p)\end{pmatrix}\right )\\ \notag
&\qquad +G_1^{(3)}(1,q)
\overline{\tau}_{\Theta_{GL_2}^{(3)},fin}\left (\begin{pmatrix} \underline{p}^i\underline{q}^j&\\ &1\end{pmatrix}
\begin{pmatrix} 1&\\ &\iota(q)\end{pmatrix}\right )\\ \notag
&\qquad+G_1^{(3)}(1,p)G_1^{(3)}(1,q)\overline{\tau}_{\Theta_{GL_2}^{(3)},fin}
\left (\begin{pmatrix}\underline{p}^i \underline{q}^j&\\
&1\end{pmatrix}\begin{pmatrix}1&\\ &\iota(p)\iota(q)\end{pmatrix}\right )\Bigg).
\end{align}
Indeed, to prove
\eqref{cfh4} one argues as in the proof of Prop.~\ref{prop61}
starting from Eqn.~\eqref{finwh4}. The integral in that equation,
adapted to our case, is given by
\begin{equation}\label{int50}
\int\limits_{F_{\nu_1}}\int\limits_{F_{\nu_2}}\widetilde{L}_\Phi(t(\underline{p}^i\underline{q}^j,1)z(r_{\nu_1})z(r_{\nu_2}))\,
\psi_{\nu_1}(r_{\nu_1})\,\psi_{\nu_2}(r_{\nu_2})\,dr_{\nu_1}\,dr_{\nu_2},
\end{equation}
where $p$ (resp.\ $q$) corresponds to place $\nu_1$ (resp.\ $\nu_2$).
To obtain \eqref{cfh4} one breaks the domain of integration in \eqref{int50} into four pieces corresponding to
$|r_{\nu_k}|\leq1$, $ |r_{\nu_k}|>1$ for $k=1,2$. Each of the four summands in equation \eqref{cfh4}
is obtained from one of the four pieces.    Eqn.~\eqref{cfh3} is similar.

To simplify \eqref{cfh3}, we may use $K_\theta$-invariance to write
$$\overline{\tau}_{\Theta_{GL_2}^{(3)},fin}\begin{pmatrix} \underline{p}^i&\\ &\iota(p)\end{pmatrix}=
\overline{\tau}_{\Theta_{GL_2}^{(3)},fin}\begin{pmatrix} \underline{p}^i&\\ &\underline{p}\end{pmatrix}.$$
Then since the central character is trivial at all archimedean places and functions in $\Theta_{GL_2}^{(3)}$ are
automorphic, in particular left-invariant under $\underline{p}I_2$, this equals
$$\overline{\tau}_{\Theta_{GL_2}^{(3)},fin}\begin{pmatrix} \underline{p}^{i-1}&\\ &1\end{pmatrix}.$$
Thus
\begin{equation}\label{cfh3-simplified}
\tau_{\sigma_{SL_2}^{(6)},fin}\begin{pmatrix} \underline{p}^i&\\ &\underline{p}^{-i}\end{pmatrix}=
\overline{\tau}_{\Theta_{GL_2}^{(3)},fin}\begin{pmatrix} \underline{p}^i&\\ &1\end{pmatrix}+G_1^{(3)}(1,p)
\overline{\tau}_{\Theta_{GL_2}^{(3)},fin}\begin{pmatrix} \underline{p}^{i-1}&\\ &1\end{pmatrix}.
\end{equation}

To simplify  \eqref{cfh4} we recall that the cocycle in
 the group $GL_2^{(3)}$ as induced from the cocycle on $Sp_4^{(3)}$ is given by the conjugate of
$$\sigma_{3,1}\left (\begin{pmatrix} a&\\ &b\end{pmatrix},\begin{pmatrix} c&\\ &d\end{pmatrix}\right )=
(ab,cd)_3(a,d)_3.$$  Thus $\overline{\tau}_{\Theta_{GL_2}^{(2)},fin}$  transforms by the adelic version of $\sigma_{3,1}$.
All Hilbert symbols in this Section are cubic;  since we will be concerned with elements of $F$ embedded in
different completions $F_\nu$, we shall sometimes write the Hilbert symbol in $F_\nu$ as $(~,~)_\nu$.

Using the right invariant property of the function $\overline{\tau}_{\Theta_{GL_2}^{(3)},fin}$ under $K_{\theta}$, we obtain
\begin{equation*}
\overline{\tau}_{\Theta_{GL_2}^{(3)},fin}\left (\begin{pmatrix} \underline{p}^i\underline{q}^j&\\ &1\end{pmatrix}
\begin{pmatrix} 1&\\ &\iota(p)\end{pmatrix}\right )=
\overline{\tau}_{\Theta_{GL_2}^{(3)},fin}\left (\begin{pmatrix} \underline{p}^i\underline{q}^j&\\ &1\end{pmatrix}
\begin{pmatrix} 1&\\ &\underline{p}\end{pmatrix}\right ).
\end{equation*}
Since we have $(\underline{p},\underline{p})=(\underline{q},\underline{p})=1$ by Hilbert reciprocity, we
may factor
$$\begin{pmatrix} \underline{p}^i\underline{q}^j&\\ &1\end{pmatrix}\begin{pmatrix} 1&\\ &\underline{p}\end{pmatrix}=\begin{pmatrix} \underline{p}&\\ &\underline{p}\end{pmatrix}
\begin{pmatrix} \underline{p}^{i-1}\underline{q}^j&\\ &1\end{pmatrix}$$
in the adelic metaplectic group (recall that the matrices are embedded via the trivial section $\mathbf{s}$).
Then we see that
$$\overline{\tau}_{\Theta_{GL_2}^{(3)},fin}\left (\begin{pmatrix} \underline{p}&\\ &\underline{p}\end{pmatrix}
\begin{pmatrix} \underline{p}^{i-1}\underline{q}^j&\\ &1\end{pmatrix}\right )
=
\overline{\tau}_{\Theta_{GL_2}^{(3)},fin}
\begin{pmatrix} \underline{p}^{i-1}\underline{q}^j&\\ &1\end{pmatrix},
$$
again using the left-invariance of functions in $\Theta_{GL_2}^{(3)}$ under $\underline{p}I_2$.

Similarly
$$\overline{\tau}_{\Theta_{GL_2}^{(3)},fin}\left (\begin{pmatrix} \underline{p}^i\underline{q}^j&\\ &1\end{pmatrix}
\begin{pmatrix} 1&\\ &\iota(q)\end{pmatrix}\right )=\overline{\tau}_{\Theta_{GL_2}^{(3)},fin}
\begin{pmatrix} \underline{p}^{i}\underline{q}^{j-1}&\\ &1\end{pmatrix}.$$

Finally, to analyze the last term in \eqref{cfh4} (where it turns out some care is needed due to the cocycle), we begin
by observing
$$\overline{\tau}_{\Theta_{GL_2}^{(3)},fin}
\left (\begin{pmatrix}\underline{p}^i \underline{q}^j&\\ &1\end{pmatrix}\begin{pmatrix}1&\\ &\iota(p)\iota(q)\end{pmatrix}\right )=
\overline{\tau}_{\Theta_{GL_2}^{(3)},fin}
\left (\begin{pmatrix}\underline{p}^i \underline{q}^j&\\ &1\end{pmatrix}\begin{pmatrix}1&\\ &\iota(p)\iota(q)\end{pmatrix}
\begin{pmatrix} 1&\\ &\iota(q)^{-1}\underline{q}\end{pmatrix}\right )$$
by the right invariance under $K_\theta$. Carrying out a cocycle
computation we obtain that the rightmost term in the above equality
is
$$(p,q)_{\nu_1} \overline{\tau}_{\Theta_{GL_2}^{(3)},fin}
\left (\begin{pmatrix}\underline{p}^i \underline{q}^j&\\ &1\end{pmatrix}\begin{pmatrix}1&\\ &\iota(p)\underline{q}\end{pmatrix}\right ).$$
In a similar way, using the right invariance by
$\left(\begin{smallmatrix} 1&\\ &\iota(p)^{-1}\underline{p}\end{smallmatrix}\right)\in K_\theta$ we find that the above expression is equal to
$$(p,q)_{\nu_1}(q,p)_{\nu_2} \overline{\tau}_{\Theta_{GL_2}^{(3)},fin}
\left (\begin{pmatrix}\underline{p}^i \underline{q}^j&\\ &1\end{pmatrix}\begin{pmatrix}1&\\ &\underline{p}\underline{q}\end{pmatrix}\right ).$$
Arguing as above, we conclude that
$$\overline{\tau}_{\Theta_{GL_2}^{(3)},fin}
\left (\begin{pmatrix}\underline{p}^i \underline{q}^j&\\ &1\end{pmatrix}\begin{pmatrix}1&\\ &\iota(p)\iota(q)\end{pmatrix}\right )=
(p,q)_{\nu_1}(q,p)_{\nu_2} \overline{\tau}_{\Theta_{GL_2}^{(3)},fin}
\begin{pmatrix}\underline{p}^{i-1} \underline{q}^{j-1}&\\ &1\end{pmatrix}.$$

Combining all this, equation \eqref{cfh4} can be written as
\begin{equation}\label{cfh41}
\tau_{\sigma_{SL_2}^{(6)},fin}\begin{pmatrix} \underline{p}^i\underline{q}^j&\\ &\underline{p}^{-i}\underline{q}^{-j}\end{pmatrix}=\gamma\big(\underline{p}^i\underline{q}^j\big)\Bigg(\overline{\tau}_{\Theta_{GL_2}^{(3)},fin}
\begin{pmatrix} \underline{p}^{i}\underline{q}^j&\\ &1\end{pmatrix}
+G_1^{(3)}(1,p)\overline{\tau}_{\Theta_{GL_2}^{(3)},fin}
\begin{pmatrix} \underline{p}^{i-1}\underline{q}^j&\\ &1\end{pmatrix}
\end{equation}
$$G_1^{(3)}(1,q)\overline{\tau}_{\Theta_{GL_2}^{(3)},fin}
\begin{pmatrix} \underline{p}^{i}\underline{q}^{j-1}&\\ &1\end{pmatrix}
+G_1^{(3)}(1,pq) \overline{\tau}_{\Theta_{GL_2}^{(3)},fin}
\begin{pmatrix}\underline{p}^{i-1} \underline{q}^{j-1}&\\ &1\end{pmatrix}\Bigg).$$
In the last term we used \eqref{gauss9}.

With this preparation, we may prove the Theorem.
The periodicity asserted in the first part follows immediately from \eqref{cfh3-simplified} and \eqref{cfh41}
together with the periodicity property for $\Theta^{(3)}_{GL_2}$. As for the second part, 
consider first the case of a single prime $p\in R$.
The identity
\begin{equation*}\label{main1}
\tau_{\sigma_{SL_2}^{(6)},fin}\begin{pmatrix} \underline{p}&\\
&\underline{p}^{-1}\end{pmatrix}= 2\,\gamma\big(\underline{p}\big)\,G_1^{(3)}(1,p)
\end{equation*}
follows directly from \eqref{cfh3} with $i=1$, since
$$\overline{\tau}_{\Theta_{GL_2}^{(3)},fin}\begin{pmatrix} \underline{p}&\\ &1\end{pmatrix}=
\overline{\tau}_{\Theta_{GL_2}^{(3)},fin}\begin{pmatrix} \iota(p)&\\ &1\end{pmatrix}=G_1^{(3)}(1,p).$$
This evaluation of the Whittaker coefficient of the cubic theta function here follows from a well-known
Hecke operator argument (Kazhdan-Patterson, \cite{K-P}, Sect.~1.4; Hoffstein \cite{H}, Prop.~5.3).

Similarly, \eqref{cfh3} with $i=2$ gives
\begin{equation*}\label{main2}
\tau_{\sigma_{SL_2}^{(6)},fin}\begin{pmatrix} \underline{p}^2&\\ &\underline{p}^{-2}\end{pmatrix}=
G_1^{(3)}(1,p)^2
\end{equation*}
since $\gamma(\underline{p}^2)=1$ and
$$
\overline{\tau}_{\Theta_{GL_2}^{(3)},fin}\begin{pmatrix} \underline{p}^2&\\ &1\end{pmatrix}=0,
$$
again via Hecke operators.  This evaluation of the coefficient $\tau_{\sigma_{SL_2}^{(6)},fin}$ is consistent
with the known properties of the Whittaker coefficient of the sextic theta function, which is again
obtained from Hecke operators.

The case of more than one prime is similar.
To prove
\begin{equation*}\label{main31}
\tau_{\sigma_{SL_2}^{(6)},fin}\begin{pmatrix} \underline{p}\,\underline{q}&\\
&\underline{p}^{-1}\underline{q}^{-1}\end{pmatrix}=
2\,\gamma\big(\underline{p}\,\underline{q}\big)\,\left(G_1^{(3)}(1,p)G_1^{(3)}(1,q)+G_1^{(3)}(1,pq)\right),
\end{equation*}
we use \eqref{cfh41} with $i=j=1$.  All terms have already been evaluated except for
$$\overline{\tau}_{\Theta_{GL_2}^{(3)},fin}\begin{pmatrix} \underline{p}\,\underline{q}&\\ &1\end{pmatrix}=\overline{\tau}_{\Theta_{GL_2}^{(3)},fin}\left (\begin{pmatrix} \iota(p)&\\ &1\end{pmatrix}\begin{pmatrix} \iota(p)^{-1}\underline{p}&\\ &1\end{pmatrix}\begin{pmatrix} \iota(q)&\\ &1\end{pmatrix}\begin{pmatrix} \iota(q)^{-1}\underline{q}&\\ &1\end{pmatrix}\right ).$$
Using invariance under $K_\theta$ and computing the cocycles, one sees that
$$\overline{\tau}_{\Theta_{GL_2}^{(3)},fin}\begin{pmatrix} \underline{p}\,\underline{q}&\\ &1\end{pmatrix}
=(p,q)_{\nu_2}\overline{\tau}_{\Theta_{GL_2}^{(3)},fin}\begin{pmatrix} \iota(p)\iota(q)&\\ &1\end{pmatrix}.
$$
This equals $G_1^{(3)}(1,pq)$, again via Hecke operators.

And to prove
\begin{equation*}
\tau_{\sigma_{SL_2}^{(6)},fin}\begin{pmatrix}
\underline{p}\,\underline{q^2}&\\
&\underline{p}^{-1}\underline{q}^{-2}\end{pmatrix}=
2\,\gamma(\underline{p})\,G_1^{(3)}(1,p)\,G_1^{(3)}(1,q)^2\left(\frac{q}{p}\right)^2_3,
\end{equation*}
(recall $\gamma(\underline{p})=\gamma(\underline{p}\,\underline{q}^2)$), we use $i=1,j=2$, recalling that
$$
\overline{\tau}_{\Theta_{GL_2}^{(3)},fin}\begin{pmatrix} \underline{p}\,\underline{q}^2&\\ &1\end{pmatrix}=0,
$$
and employing \eqref{gauss9}, the link between local Hilbert symbols and global power residue symbols,
and reciprocity.  (Note $(p,q)_v=1$ for all $v\in S$ since $p,q\in R$.)

This completes the proof of the theorem. \end{proof}

{\bf Remarks.}  {\bf 1.}  Let $Z$ denote the center of $GL_2$.
Using the Kubota cocycle for $GL_2$ as in \cite{K-P}, pg.\ 41, it is easy to see if $F$ is a local field
containing enough roots of unity then for a metaplectic cover of any degree,
$\widetilde{SL_2}(F)$ and $\widetilde{Z}(F)$ commute in $\widetilde{GL_2}(F)$.  Moreover for the double
cover $\widetilde{Z}^{(2)}$ is split with section $\gamma$.  Using this, the descent integral and Theorem~\ref{maincfh}
may easily be extended to $\widetilde{Z}\cdot\widetilde{SL_2}\subseteq GL_2^{(6)}$, giving 
 the formula predicted in \cite{C-F-H}, pg.\ 155 for $m_1,m_2\in R$.

{\bf 2.} The numerical work of \cite{Br-Ho} detects a sixth root of unity in the Whittaker coefficients of certain
vectors in the theta space which does not appear here; however
here $\tau_{\text{fin}}$ includes a contribution from the places in $S$ while the coefficients computed in \cite{Br-Ho}
do not.   Also, \cite{Br-Ho} consider specific vectors and the authors there note that some choices give cleaner results than others.
By contrast, our result holds when the vector $\theta$ is chosen so that the conditions above apply, and in particular \eqref{cfh2} is 1.  We
are not able to specify the exact vector satisfying this condition.  Also,  our result applies to primes satisfying certain 
congruence conditions.

\section{Descent Integrals and Patterson's Conjecture}

The situation in the $SO_7$ case is similar. Recall that for the quartic theta function, the coefficients satisfy
a periodicity property modulo fourth powers, the coefficients
at cubes of primes are zero and those at squares are known.  Patterson's conjecture concerns
the value at square-free arguments.  We have constructed a descent function on the 4-fold cover of
$SO_4$. Via
the identification of $D_2$ with $A_1\times A_1$, this function may be regarded as a function on a cover of $GL_2$.

More precisely, let
$$(GL_2\times GL_2)^0=\{ (g_1,g_2)\in GL_2\times GL_2\ :\ \text{det}
g_1=\text{det} g_2\}.$$
Then there is a natural map from $(GL_2\times GL_2)^0$ to $SO(4)$ as in \cite{Bump}, Ch.\ 30,
with kernel $Z$ consisting of scalar matrices in $GL_2$ embedded diagonally.
Under this map, the torus $\text{diag}
(a,1,1,a^{-1})\in SO_4$ may be identified with the torus consisting of pairs
$\left(\left(\begin{smallmatrix}a&\\&1\end{smallmatrix}\right),\left(\begin{smallmatrix}a&\\&1\end{smallmatrix}\right)\right)$
in $Z\backslash (GL_2\times GL_2)^0.$
This identification identifies the coefficients of the descent on the 4-fold cover of
$SO_4$ with {\it squares} of the corresponding coefficients of a function on the 4-fold cover of $GL_2$.

We take a specific vector $\theta$ similar to the one described above in the $Sp_4$ case.  In particular
we choose $\theta$ so that $\tau_{\sigma_{SO_4}^{(4)},fin}(e)=1$.  Let $S$ be a finite set of places such that
everything is unramified outside $S$, and define the set $R$
similarly to the $Sp_4$ case.   Let $m(a)=\text{diag}(a,1,1,a^{-1})\in SO_4$.  We have
\begin{theorem}\label{mainso7}
Let $\tau_{\sigma_{S0_4}^{(4)},fin}$ be the finite part of the
normalized Whittaker coefficient attached to the vector $\theta\in \sigma_{SO_4}^{(4)}$.
\begin{enumerate}
\item Let $a,b\in R$.  Then
$$\tau_{\sigma_{S0_4}^{(4)},fin}(m(\underline{a}\,\underline{b}^4))= |b|\,
\tau_{\sigma_{S0_4}^{(4)},fin}(m(\underline{a})).
$$
\item Let $a\in R$ be square-free.  Then
$$\tau_{\sigma_{S0_4}^{(4)},fin}(m(\underline{a}))=G_3^{(4)}(1,a)\sum_{\substack{d_1,d_2\in R\\d_1d_2=a}}\left(\frac{d_1}{d_2}\right)_2.
$$
\end{enumerate}
\end{theorem}

\begin{proof}
Once again for notational convenience we consider prime and coprime elements $\underline{p},\underline{q}$ in $R$,
corresponding to distinct places $\nu_1$,  $\nu_2$ and consider $a=p^iq^j$, $i,j=0,1$.  These arguments extend to more
than two primes without difficulty. 

The starting point is the proof of Prop.~\ref{prop62}. Indeed, arguing as in the $Sp_4$ case we obtain the following two identities. First, for integers $i\ge 1$ we have
\begin{multline*}
\tau_{\sigma_{S0_4}^{(4)},fin}(m(\underline{p}^i))=\\ \tau_{\Theta_{GL_3}^{(4)},fin}\begin{pmatrix} \underline{p}^i&&\\ &1&\\ &&1\end{pmatrix}+G_3^{(4)}(1,p)\tau_{\Theta_{GL_3}^{(4)},fin}\left (\begin{pmatrix} \underline{p}^i&&\\ &1&\\ &&1\end{pmatrix}
\begin{pmatrix} 1&&\\ &\iota(p)&\\ &&\iota(p)\end{pmatrix}\right ).
\end{multline*}
Here $m(a)=\text{diag}(a,1,1,a^{-1})\in SO_4$. The second identity, valid for all integers $i,j\ge 1$, is
\begin{align}\label{so721}
\tau_{\sigma_{S0_4}^{(4)},fin}(m(\underline{p}^i\underline{q}^j))&= \tau_{\Theta_{GL_3}^{(4)},fin}\begin{pmatrix} \underline{p}^i\underline{q}^j&&\\ &1&\\ &&1\end{pmatrix}\\ \notag
&\quad+G_3^{(4)}(1,p)\tau_{\Theta_{GL_3}^{(4)},fin}\left (\begin{pmatrix} \underline{p}^i\underline{q}^j&&\\ &1&\\ &&1\end{pmatrix}
\begin{pmatrix} 1&&\\ &\iota(p)&\\ &&\iota(p)\end{pmatrix}\right ) \\ \notag
&\quad
+G_3^{(4)}(1,q)\tau_{\Theta_{GL_3}^{(4)},fin}\left (\begin{pmatrix} \underline{p}^i\underline{q}^j&&\\ &1&\\ &&1\end{pmatrix}
\begin{pmatrix} 1&&\\ &\iota(q)&\\ &&\iota(q)\end{pmatrix}\right )\\ \notag
&
+G_3^{(4)}(1,p)G_3^{(4)}(1,q)\tau_{\Theta_{GL_3}^{(4)},fin}\left (\begin{pmatrix} \underline{p}^i\underline{q}^j&&\\ &1&\\ &&1\end{pmatrix}
\begin{pmatrix} 1&&\\ &\iota(p)\iota(q)&\\ &&\iota(p)\iota(q)\end{pmatrix}\right ).
\end{align}

Using the fact that $(\underline{p},\underline{p})=(\underline{q},\underline{p})=1$, we argue as in the proof of
Thm.~\ref{maincfh} to obtain the identities
\begin{equation}\label{so731}
\tau_{\sigma_{S0_4}^{(4)},fin}(m(\underline{p}^i))= \tau_{\Theta_{GL_3}^{(4)},fin}\begin{pmatrix} \underline{p}^i&&\\ &1&\\ &&1\end{pmatrix}+G_3^{(4)}(1,p)\tau_{\Theta_{GL_3}^{(4)},fin}\begin{pmatrix} \underline{p}^{i-1}&&\\ &1&\\ &&1\end{pmatrix}
\end{equation}
and
\begin{align}\label{so741}
\tau_{\sigma_{S0_4}^{(4)},fin}(m(\underline{p}^i\underline{q}^j))&=\tau_{\Theta_{GL_3}^{(4)},fin}\begin{pmatrix} \underline{p}^i\underline{q}^j&&\\ &1&\\ &&1\end{pmatrix} 
+G_3^{(4)}(1,p)\tau_{\Theta_{GL_3}^{(4)},fin}\begin{pmatrix} \underline{p}^{i-1}\underline{q}^j&&\\ &1&\\ &&1\end{pmatrix}\\ \notag
&\qquad +G_3^{(4)}(1,q)\tau_{\Theta_{GL_3}^{(4)},fin}\begin{pmatrix} \underline{p}^i\underline{q}^{j-1}&&\\ &1&\\ &&1\end{pmatrix}
\\ \notag
&\qquad
+G_3^{(4)}(1,pq)\tau_{\Theta_{GL_3}^{(4)},fin}\begin{pmatrix} \underline{p}^{i-1}\underline{q}^{j-1}&&\\ &1&\\ &&1\end{pmatrix}.
\end{align}

In deriving \eqref{so741} from \eqref{so721} the only difficulty is to check that we get the right cocycle
so as to obtain the last summand of  \eqref{so741}.  To explain this point,
starting with the last summand of  \eqref{so721}, by $K_\theta$-invariance we have
\begin{multline*}
\tau_{\Theta_{GL_3}^{(4)},fin}\left (\begin{pmatrix} \underline{p}^i\underline{q}^j&&\\ &1&\\ &&1\end{pmatrix}
\begin{pmatrix} 1&&\\ &\iota(p)\iota(q)&\\ &&\iota(p)\iota(q)\end{pmatrix}\right )\\ =
\tau_{\Theta_{GL_3}^{(4)},fin}\left (\begin{pmatrix} \underline{p}^i\underline{q}^j&&\\ &1&\\ &&1\end{pmatrix}
\begin{pmatrix} 1&&\\ &\iota(p)\iota(q)&\\ &&\iota(p)\iota(q)\end{pmatrix}\begin{pmatrix} 1&&\\ &\iota(q)^{-1}\underline{q}&\\ &&\iota(q)^{-1}\underline{q}\end{pmatrix}\right ).
\end{multline*}
Multiplying the last two matrices, the cocycle contributes a factor of
$(p,q)_{\nu_1}^3$, where $(~,~)_{\nu_1}$ is the local 4-th order residue symbol in $F_{\nu_1}$.
Adjusting on the right by a similar factor involving $p$, we obtain
\begin{multline*}
\tau_{\Theta_{GL_3}^{(4)},fin}\left (\begin{pmatrix} \underline{p}^i\underline{q}^j&&\\ &1&\\ &&1\end{pmatrix}
\begin{pmatrix} 1&&\\ &\iota(p)\iota(q)&\\ &&\iota(p)\iota(q)\end{pmatrix}\right )
=\\(p,q)_{\nu_1}^3(q,p)_{\nu_2}^3\tau_{\Theta_{GL_3}^{(4)},fin}\left (\begin{pmatrix} \underline{p}^i\underline{q}^j&&\\ &1&\\ &&1\end{pmatrix}
\begin{pmatrix} 1&&\\ &\underline{p}\,\underline{q}&\\ &&\underline{p}\,\underline{q}\end{pmatrix}\right ).
\end{multline*}
Thus by \eqref{gauss9} we obtain the last term in \eqref{so741}.  

The first part of the Theorem now follows
from \eqref{so731}, \eqref{so741} by
 the periodicity properties of $\Theta_3^{(4)}$. Note that we obtain the factor $|b|$ from the periodicity 
for $\Theta_3^{(4)}$ since this function is on 
a cover of $GL_3$
(apply \cite{Ho}, Prop.\ 5.1, with $n=4$, $r=3$).   As for the second part, 
the desired conclusion now
follows by putting $i,j=0,1$ and using the values of $\tau_{GL_3^{(4)},fin}$.  Note that we have
$$G_3^{(4)}(1,p)G_3^{(4)}(1,q)=\left(\frac{q}{p}\right)_2 G_3^{(4)}(1,pq),$$
since under our assumptions
$$\left(\frac{q}{p}\right)_4=\left(\frac{p}{q}\right)_4.$$
\end{proof}

The first part of Theorem~\ref{mainso7} is consistent with the periodicity property for the square of the coefficients
of a $GL_2$ theta function, and the second part 
is consistent with the Patterson conjecture, in the sense that the 
Fourier coefficients of the
square-integrable function on the four-fold cover that we have obtained as a descent
satisfy the conjectured relation for the quartic theta coefficients for these primes.
We also remark that Eqn.~\eqref{so731} with $i=2$ gives
$$\tau_{\sigma_{S0_4}^{(4)},fin}(m(\underline{p^2}))=G_3^{(4)}(1,p)^2.
$$
This is also consistent with Conjecture~\ref{conjecture1}, since the corresponding fact for the quartic
theta function is known via Hecke theory (see for example\ \cite{C-F-H}, pg.\ 153).

\end{document}